\documentclass[11pt,reqno]{article}

\usepackage{caption,amsmath,amssymb,amsfonts,amsthm,latexsym,graphicx,multirow,color,hyperref,enumerate}
\usepackage{tikz}
\usetikzlibrary{positioning}

\newtheorem{theorem}{Theorem}[section]

\newtheorem{lemma}[theorem]{Lemma}
\newtheorem{proposition}[theorem]{Proposition}
\newtheorem{problem}[theorem]{Problem}

\theoremstyle{definition}
\newtheorem{definition}[theorem]{Definition}

\newtheorem*{remark}{Remark}
\newtheorem*{notation}{Notation}

\numberwithin{equation}{section}
\numberwithin{table}{section}

\newcommand{\magma}{\textsc{Magma}}
\renewcommand{\a}{\alpha}

\newcommand{\Aut}{\mathrm{Aut}}
\newcommand{\Cay}{\mathrm{Cay}}

\begin{document}
	
	\title{The existence of $m$-Haar graphical representations}
	
	\author{
		Jia-Li Du\thanks{Corresponding author, School of Mathematical Sciences, Nanjing Normal University, Nanjing, 210023, P.R. China}
			\thanks{ Ministry of Education Key Laboratory of NSLSCS, Nanjing, 210023, P.R. China},
		Yan-Quan Feng\thanks{School of Mathematics and Statistics, Beijing Jiaotong University, Beijing 100044, China},
		Binzhou Xia\thanks{School of Mathematics and Statistics, The University of Melbourne, Parkville, VIC 3010, Australia},
		Da-Wei Yang\thanks{School of Mathematical Sciences, Beijing University of Posts and Telecommunications, Beijing, 100876, P.R. China}
		\thanks{Key Laboratory of Mathematics and Information Networks (BUPT), Ministry of Education, Beijing, 100876, P.R. China}
	}
	
	\date{}
	\maketitle
	
	\begin{abstract}
		Extending the well-studied concept of graphical regular representations to bipartite graphs, a Haar graphical representation (HGR) of a group $G$ is a bipartite graph whose automorphism group is isomorphic to $G$ and acts semiregularly with the orbits giving the bipartition. The question of which groups admit an HGR was inspired by a closely related question of Est\'elyi and Pisanski in 2016, as well as Babai's work in 1980 on poset representations, and has been recently solved by Morris and Spiga. In this paper, we introduce the $m$-Haar graphical representation ($m$-HGR) as a natural generalization of HGR to $m$-partite graphs for $m\geq2$, and explore the existence of $m$-HGRs for any fixed group. This inquiry represents a more robust version of the existence problem of G$m$SRs as addressed by Du, Feng and Spiga in 2020. Our main result is a complete classification of finite groups $G$ without $m$-HGRs.
		
		\emph{Keywords}: Haar graphical representation; $m$-Cayley graph; $m$-Haar graphical representation.
		
		\emph{MSC2020}: 05C25, 20B25.
	\end{abstract}

	\section{Introduction}
	
	With a rich history of research and a wide range of variations, the GRR problem stands as a classical and well-established topic at the intersection of graph theory and group theory. A \emph{graphical regular representation} (GRR) of a group $G$ is a graph $\Gamma$ whose automorphism group $\Aut(\Gamma)$ is isomorphic to $G$ and acts regularly on the vertex set of $\Gamma$.
	The investigation into which groups have GRRs dates back to 1964, when Chao~\cite{Chao} and Sabidussi~\cite{Sabidussi1964} independently proved that there are no GRRs of finite abelian groups with exponent greater than $2$. This finding then spurred numerous papers on the classification of finite groups admitting GRRs, eventually completed by Godsil~\cite{Godsil} with significant contributions from others such as Hetzel~\cite{Hetzel}, Imrich~\cite{Imrich}, Nowitz and Watkins~\cite{NowitzWatkins1,NowitzWatkins2}.
	
	A graph $\Gamma$ is a Cayley graph of a group $G$ if and only if $\Aut(\Gamma)$ has a regular subgroup isomorphic to $G$. Thus GRRs of $G$ are precisely Cayley graphs of $G$ whose automorphism group is isomorphic to $G$.
	As a generalization, for any integer $m\geq1$, a \emph{graphical $m$-semiregular representation} (G$m$SR) of $G$ is defined in~\cite{DFS} to be an $m$-Cayley graph of $G$ whose automorphism group is isomorphic to $G$, where an $m$-Cayley graph of $G$ is by definition (see~\cite[Page~353]{AKMM2011}) a graph $\Gamma$ such that $\Aut(\Gamma)$ has a semiregular subgroup isomorphic to $G$ with $m$ orbits on the vertex set.
	G$m$SRs are also referred to as $m$-GRRs in~\cite{HKM}, where it is shown that every finite group admitting a GRR also admits a G$m$SR for $m\geq2$. This result is strengthened to assert the existence of regular (every vertex having the same valency) G$m$SRs by Du, Feng and Spiga~\cite{DFS}, who were then able to prove the following theorem based on the classification of finite groups admitting GRRs.
	
	\begin{theorem}
		\emph{(\cite[Theorem~1.1]{DFS})} Let $G$ be a finite group and $m\geq2$ be an integer. Then $G$ has a regular G$m$SR if and only if none of the following occurs:
		\begin{enumerate}[\rm(a)]
			\item $m=2$, and $G=C_n$ with $n\in\{1,2,3,4,5\}$, $C_2^2$ or $Q_8$;
			\item $m=3$, and $G=C_n$ with $n\in\{1,2,3\}$;
			\item $m=4$, and $G=C_n$ with $n\in\{1,2\}$;
			\item $5\leq m\leq9$, and $G=1$.
		\end{enumerate}
	\end{theorem}
	
	Following the approach to the above theorem in~\cite{DFS}, edges within $G$-orbits are indispensable in the machinery to control the automorphism group except for $G=C_1$. This causes the produced regular G$m$SRs to be not $m$-partite when $G\neq C_1$ and poses challenges in the construction of the designated G$m$SRs as described below.
	
	\begin{definition}\label{DefmHGR}
		For an integer $m\geq2$, a graph $\Gamma$ is called an \emph{$m$-Haar graphical representation} ($m$-HGR) of a group $G$ if $\Gamma$ is regular and $m$-partite such that $\Aut(\Gamma)$ is isomorphic to $G$ and acts semiregularly on the vertex set with orbits giving the $m$-partition.
	\end{definition}
	
	The terminology $m$-Haar graphical representation originates from the \emph{Haar graphical representation} (HGR) coined in~\cite{DFS2}.
	Initially introduced in~\cite{HMP} in the language of voltage graphs, a \emph{Haar graph} of a group $G$ is a bipartite graph whose automorphism group has a subgroup that is isomorphic to $G$ and semiregular on the vertex set with orbits giving the bipartition. Notice that Haar graphs are necessarily regular. Thus an HGR of a group $G$, or equivalently a $2$-HGR of $G$, is precisely a Haar graph of $G$ whose automorphism group is isomorphic to $G$.
	It is also easy to see that each Haar graph on an abelian group is Cayley, whence abelian groups have no HGRs.
	In 2016, Est\'elyi and Pisanski~\cite{Estelyi} asked to determine the class $\mathcal{C}$ of finite nonabelian groups $G$ all of whose Haar graphs are Cayley.
	
	The question of Est\'elyi and Pisanski was answered for dihedral groups~\cite{Estelyi}, inner abelian groups~\cite{FKY} and nonsolvable groups~\cite{FKY}, before a complete resolution by Feng, Kov\'acs, Wang and Yang~\cite{FKWY}.
	Noting by definition that no groups in $\mathcal{C}$ admit HGRs, the authors in~\cite{FKWY} then proposed the following problem.
	
	\begin{problem}\label{prob=1}
		Determine the finite nonabelian groups without HGRs.
	\end{problem}
	
	For an even integer $n\geq6$, denote by $D_{n}$ the dihedral group of order $n$.
	Recently, Problem~\ref{prob=1} was completely solved by Morris and Spiga~\cite{MS}, who established Theorem~\ref{theo=HGR} below.
	Their motivation partially comes from the so-called \emph{poset representation} of a group $G$, which is defined to be a poset $P$ with $\Aut(P)\cong G$. In 1980, Babai applied results on digraphical regular representations to prove that every group $G$ of order $n$ admits a poset representation of order at most $3n$~\cite[Corollary~4.3]{Babai}. This is improved to that all but finitely many $G$ admits a poset representation of order at most $2n$~\cite[Corollary~1.4]{MS}, as a consequence of Theorem~\ref{theo=HGR}.
	
	\begin{theorem}\label{theo=HGR}\emph{(\cite[Theorem~1.1]{MS})}
		A finite nonabelian group $G$ does not have a HGR if and only if $G$
		is isomorphic to $D_n$ with $n\in \{6, 8, 10, 12, 14\}$, $Q_8$, $A_4$, $Q_8\times C_2$, $\langle x,y~|~x^6=y^4=1, x^3=y^2,x^y=x^{-1}\rangle$ or $\langle x,y,z~|~x^3=y^3=z^2=[x,y]=1, x^z=x^{-1}, y^z=y^{-1}\rangle$.
	\end{theorem}
	
	The main result of this paper extends Theorem~\ref{theo=HGR} to $m$-HGRs.
	
	\begin{theorem}\label{theo=main}
		Let $G$ be a finite group and $m\geq3$ be an integer. Then $G$ has an $m$-HGR if and only if none of the following occurs:
		\begin{enumerate}[\rm(a)]
			\item $m=3$, and $G=C_1$, $C_2$, $C_3$, $C_4$, $C_5$, $C_2^2$ or $D_6$;
			\item $m=4$, and $G=C_1$, $C_2$ or $C_3$;
			\item $m=5$, and $G=C_1$ or $C_2$;
			\item $6\leq m\leq 9$, and $G=C_1$.
		\end{enumerate}
	\end{theorem}
	
	The the paper is structured as follows. After this Introduction, we establish in Section~\ref{Sec2} the necessary terminology and notations, followed by the key idea of our approach. In Sections~\ref{Sec3} and~\ref{Sec4}, we treat groups with a small generating set. With these preparations, we finally prove Theorem~\ref{theo=main} for odd $m$ and even $m$, respectively, in Section~\ref{Sec5} and~\ref{Sec6}. For the rest of this paper, all groups and graphs are assumed to be finite.

	\section{Approach}\label{Sec2}
	
	Let $\Gamma$ be a graph. Denote by $V(\Gamma)$ and $E(\Gamma)$ the vertex set and edge set of $\Gamma$, respectively.
	For a subset $S$ of $V(\Gamma)$, denote by $\Gamma[S]$ the subgraph induced by $S$ in $\Gamma$. When the graph $\Gamma$ is clear from the context, we simply write $[S]$ for $\Gamma[S]$. For $v\in V(\Gamma)$, let $\Gamma(v)$ denote the neighborhood (set of neighbors) of $v$ in $\Gamma$, and let $\Gamma(\Gamma(v))$ denote the set of neighbors of the vertices in $\Gamma(v)$.
	
	Let $G$ be a group. Denote by $d(G)$ the smallest size of a generating set of $G$. For $g\in G$, denote the order of $g$ by $|g|$. If $G$ acts on a set $\Omega$, then for $\omega \in \Omega$, denote by $G_\omega$ the stabilizer in $G$ of $\omega$, that is, $G_\omega$ denotes the subgroup of $G$ that consists of elements fixing $\omega$. We say that $G$ is \emph{semiregular} on $\Omega$ if $G_\omega=1$ for all $\omega \in \Omega$. Moreover, $G$ is \emph{regular} if it is semiregular and transitive.
	
	For convenience in the subsequent discussion, we introduce some terminology and notations related to $m$-Cayley graphs of the group $G$.
	For a subset $T$ of $G$, denote
	\[
	T^{-1}=\{t^{-1}~|~t\in T\}.
	\]
	Now let $m$ be a positive integer, and let $(T_{i,j})_{m\times m}$ be an $m\times m$ matrix with entries $T_{i,j}\subseteq G$ satisfying both of the conditions:
	\begin{itemize}
		\item $1\notin T_{i,i}$ for all $i\in\{1,\dots,m\}$;
		\item $T_{j,i}=T_{i,j}^{-1}$ for all $i,j\in\{1,\dots,m\}$.
	\end{itemize}
	Then the \emph{$m$-Cayley graph} $\Cay(G,(T_{i,j})_{m\times m})$ of $G$ with \emph{connection matrix} $(T_{i,j})_{m\times m}$ is the graph with vertex set
	\[
	G\times\{1,\dots,m\}=\bigcup_{i\in\{1,\dots,m\}}G_i,
	\]
	where $G_i=G\times\{i\}=\{g_i~|~ g\in G\}$ and $g_i=(g,i)$, and edge set
	\begin{equation}\label{EqDef}
		\bigcup_{i,j\in\{1,\dots,m\}}\big\{\{g_i, (tg)_j\}~|~g\in G,\,t\in T_{i,j}\big\}.
	\end{equation}
	
	\begin{notation}
		Throughout the paper, $g_i$ is a concise notation for $(g,i)$ if $g$ is an element of a group and $i$ is an positive integer (we will still use $(g,i)$ sometimes if it makes the context clearer). Similarly, for a set $S$ of elements, $S\times\{i\}$ can be concisely denoted as $S_i$.
	\end{notation}
	
	The $1$-Cayley graph of $G$ with connection matrix $(T)$ is exactly the Cayley graph of $G$ with connection set $T$. A regular $m$-Cayley graph $\Cay(G,(T_{i,j})_{m\times m})$ with $m\geq2$ and $T_{1,1}=\dots=T_{m,m}=\emptyset$ is called an \emph{$m$-Haar graph}.
	
	\begin{remark}
		By the condition $T_{j,i}=T_{i,j}^{-1}$ in the definition of a connection matrix $(T_{i,j})_{m\times m}$, the sets $T_{j,i}$ with $j>i$ are determined by the sets $T_{i,j}$. Moreover, the definition of an $m$-Haar graph requires $T_{1,1}=\dots=T_{m,m}=\emptyset$. Thus, to give an $m$-Haar graph, we only need to specify the sets $T_{i,j}$ for $1\leq i<j\leq m$.
	\end{remark}
	
	For each $g\in G$, the right multiplication of $g$ induces an automorphism
	\[
	R_m(g): G\times\{1,\dots,m\}\to G\times\{1,\dots,m\},\ \ x_i\mapsto (xg)_i
	\]
	of each $m$-Cayley graph $\Cay(G,(T_{i,j})_{m\times m})$. Clearly, $R_m(G):=\{R_m(g):g\in G\}$ is a subgroup of the automorphism group of $\Cay(G,(T_{i,j})_{m\times m})$, which is semiregular on the vertex set with orbits $G_1,\dots,G_m$. Since $R_m$ is a group isomorphism from $G$ to $R_m(G)$, we identify $R_m(G)$ with $G$ when there is no confusion.
	Also, it is easy to see that a graph is (isomorphic to) an $m$-Cayley graph of $G$ if and only if the graph has a semiregular group of automorphisms isomorphic to $G$ with $m$ orbits.
	Hence Definition~\ref{DefmHGR} leads to the following fact.
	
	\begin{lemma}
		For a group $G$, the $m$-HGRs of $G$ are the $m$-Haar graphs of $G$ whose automorphism group is isomorphic to $G$.
	\end{lemma}
	
	For a (not necessarily regular) $m$-Cayley graph $\Gamma=\Cay(G,(T_{i,j})_{m\times m})$ with $T_{1,1}=\dots=T_{m,m}=\emptyset$, if $\Aut(\Gamma)\cong G$, then $\Gamma$ is called an \emph{$m$-partite graphical semiregular representation} ($m$-PGSR) of $G$.
	The key idea in the proof of Theorem~\ref{theo=main} is to construct $m$-HGRs from small $\ell$-PGSRs, which is facilitated by the following three technical lemmas.
	
	\begin{lemma}\label{Lem3PGSR}
		Let $G$ be a finite group, and let $\Sigma$ be a $3$-PGSR of $G$ such that the following conditions are satisfied for some integer $k$ with $2\leq k\leq|G|$:
		\begin{enumerate}[\rm(a)]
			\item the vertex $(g,1)$ has valency $k+1$ for each $g\in G$;
			\item the vertices $(g,2)$ and $(g,3)$ have valency $k$ for each $g\in G$;
			\item there exists a $3$-cycle through $(g,i)$ for each $g\in G$ and $i\in\{1,2,3\}$.
		\end{enumerate}
		Then $G$ has an $m$-HGR for each odd integer $m\geq 5$.
	\end{lemma}
	
	\begin{proof}
		Write $\Sigma=\Cay(G,(T_{i,j})_{3\times 3})$ with $T_{1,2}=S$, $T_{1,3}=L$ and $T_{2,3}=T$. According to conditions~(a) and~(b), we have $|S|+|L|=k+1$ and $|S|+|T|=|L|+|T|=k$.
		Take any subsets $M$ and $N$ of $G$ such that $|M|=k-1$ and $|N|=k$.
		Let $\Gamma_m=\Cay(G,(T_{i,j})_{m\times m})$ with $T_{1,2}=S$, $T_{1,3}=L$, $T_{2,3}=T$, $T_{m-1,m}=N$ and
		\[
		T_{i,j}=
		\begin{cases}
			M&\text{if }j=i+1\text{ is odd and }4\leq i\leq m-3,\\
			\{1\}&\text{if }j=i+2\text{ and }2\leq i\leq m-2,\\
			\emptyset&\text{otherwise}.
		\end{cases}
		\]
		The structure of $\Gamma_m$ is illustrated in Figure~\ref{Fig1}. Note that $\Gamma_m$ is $(k+1)$-regular since
		\[
		|S|+|L|=|S|+|T|+1=|L|+|T|+1=|M|+2=|N|+1=k+1.
		\]
		Hence $\Gamma_m$ is an $m$-Haar graph of $G$. Let $A=\Aut(\Gamma_m)$.
		
		\begin{figure}[!hhh]
			\vspace{3mm}
			\caption{From $3$-PGSR to $m$-HGR for odd $m\geq 5$}\label{Fig1}
			\begin{center}
				\begin{tikzpicture}[node distance=1cm,thick,scale=0.75,every node/.style={transform shape}]
					\node[circle](C0){};
					\node[left=of C0,circle,draw,inner sep=11.4pt, label=45:{}](G2){$G_{2}$};
					\node[below=of G2,circle,draw,inner sep=11.5pt,label=-90:{}](G3){$G_{3}$};
					\node[left=of G3,circle,draw,inner sep=11.3pt,label=-90:{}](G1){$G_1$};
					\node[right=of G2,circle,draw,inner sep=10.1pt,label=-90:{}](G4){$G_{2s}$};
					\node[below=of G4,circle,draw,inner sep=7.1pt,label=-90:{}](G5){$G_{2s+1}$};
					\node[right=of G5,circle,draw,inner sep=11pt,label=-90:{}](Gm){$G_{m}$};
					\node[right=of G4,circle,draw,inner sep=7.2pt,label=-90:{}](Gm-1){$G_{m-1}$};
					\draw(Gm) to node[right]{$N$} (Gm-1);
					\draw[dashed](Gm) to node[below]{$\{1\}$} (G5);
					\draw[dashed](Gm-1) to node [above]{$\{1\}$}(G4);
					\draw(G4) to node[right]{$M$}(G5);
					\draw[dashed](G2) to node [above]{$\{1\}$}(G4);
					\draw[dashed](G3) to node[below]{$\{1\}$}(G5);
					\draw(G1) to node[above]{$S$}(G2);
					\draw(G1) to node[below]{$L$} (G3);
					\draw(G2) to node[right]{$T$}(G3);
				\end{tikzpicture}
			\end{center}
		\end{figure}
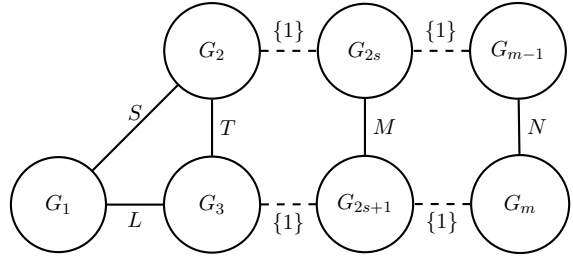
		
		By condition~(c), there exists a $3$-cycle through $(g,i)$ for each $g\in G$ and $i\in\{1,2,3\}$.
		On the other hand, for each $i\in\{4,\dots,m\}$, there is no $3$-cycle through $(g,i)$ for any $g\in G$.
		Hence $A$ stabilizes $G_1\cup G_2\cup G_3$. Since the induced subgraph $[G_1\cup G_2\cup G_3]\cong \Sigma$ is a $3$-PGSR of $G$, it follows that $A$ stabilizes $G_1$, $G_2$ and $G_3$ respectively, and that $A_{1_1}$ fixes $G_1\cup G_2\cup G_3$ pointwise. Since $T_{2,4}=T_{3,5}=\{1\}$ and $T_{i,j}=\emptyset$ for pairs $(i,j)$ with $i \leq3$ and $j\geq 4$ other than $(2,4)$ and $(3,5)$, it follows that $A$ stabilizes $G_4$ and $G_5$ respectively, and that $A_{1_1}$ fixes $G_4\cup G_5$ pointwise. Repeating this argument, we derive that $A$ stabilizes $G_i$ for each $i\in\{1,\dots,m\}$, and $A_{1_1}$ fixes $G_1\cup\dots\cup G_{m}$ pointwise. This shows that $\Gamma_m$ is an $m$-HGR of $G$.
	\end{proof}
	
	\begin{lemma}\label{Lem4PGSR}
		Let $G$ be a finite group, and let $\Sigma$ be a $4$-PGSR of $G$ such that the following conditions are satisfied for some integer $k$ with $2\leq k\leq|G|$:
		\begin{enumerate}[\rm(a)]
			\item the vertices $(g,1)$ and $(g,2)$ have valency $k+1$ for each $g\in G$;
			\item the vertices $(g,3)$ and $(g,4)$ have valency $k$ for each $g\in G$;
			\item there exists a $3$-cycle through $(g,i)$ for each $g\in G$ and $i\in\{1,2,3,4\}$.
		\end{enumerate}
		Then $G$ has an $m$-HGR for each even integer $m\geq 6$.
	\end{lemma}
	
	\begin{proof}
		Write $\Sigma=\Cay(G,(T_{i,j})_{4\times 4})$ with $T_{1,2}=S$, $T_{1,3}=L$, $T_{1,4}=K$, $T_{2,3}=R$, $T_{2,4}=T$ and $T_{3,4}=U$. Then by conditions~(a) and~(b), we have $|S|+|L|+|K|=|R|+|T|+|S|=k+1$ and $|L|+|R|+|U|=|K|+|T|+|U|=k$. Take any subsets $M$ and $N$ of $G$ of size $k-1$ and $k$, respectively. Let $\Gamma_m=\Cay(G,(T_{i,j})_{m\times m})$ with $T_{1,2}=S$, $T_{1,3}=L$, $T_{1,4}=K$, $T_{2,3}=R$, $T_{2,4}=T$, $T_{3,4}=U$, $T_{m-1,m}=N$ and
		\[
		T_{i,j}=
		\begin{cases}
			M&\text{if }j=i+1\text{ is even and }5\leq i\leq m-3,\\
			\{1\}&\text{if }j=i+2\text{ and }3\leq i\leq m-2,\\
			\emptyset&\text{otherwise}.
		\end{cases}
		\]
		The structure of $\Gamma_m$ is illustrated in Figure~\ref{Fig2}. Since
		\begin{align*}
			|S|+|L|+|K|=|R|+|T|+|S|&=|L|+|R|+|U|+1\\
			&=|K|+|T|+|U|+1=|M|+2=|N|+1=k+1,
		\end{align*}
		it follows that $\Gamma_m$ is $(k+1)$-regular and hence an $m$-Haar graph of $G$. Let $A=\Aut(\Gamma_m)$.
		
		\begin{figure}[!hhh]
			\vspace{3mm}
			\caption{From $4$-PGSR to $m$-HGR for even $m\geq 6$}\label{Fig2}
			\begin{center}
				\begin{tikzpicture}[node distance=1cm,thick,scale=0.75,every node/.style={transform shape}]
					\node[circle](C0){};
					\node[right=of C0,circle,draw,inner sep=11pt, label=45:{}](H1){$G_{1}$};
					\node[below=of H1,circle,draw,inner sep=11pt,label=-90:{}](H2){$G_{2}$};
					\node[right=of H1,circle,draw,inner sep=11pt,label=-90:{}](H3){$G_3$};
					\node[right=of H2,circle,draw,inner sep=11pt,label=-90:{}](H4){$G_{4}$};
					\node[right=of H4,circle,draw,inner sep=9.8pt,label=-90:{}](H2k){$G_{2s}$};
					\node[right=of H3,circle,draw,inner sep=6.2pt,label=-90:{}](H2k-1){$G_{2s-1}$};
					\node[right=of H2k-1,circle,draw,inner sep=6.5pt,label=-90:{}](Hm-1){$G_{m-1}$};
					\node[right=of H2k,circle,draw,inner sep=9.8pt,label=-90:{}](Hm){$G_{m}$};
					\draw(Hm) to node[right]{$N$} (Hm-1);
					\draw[dashed](Hm) to node[below]{$\{1\}$} (H2k);
					\draw[dashed](Hm-1) to node [above]{$\{1\}$}(H2k-1);
					\draw(H2k) to node[right]{$M$}(H2k-1);
					\draw[dashed](H4) to node [below]{$\{1\}$}(H2k);
					\draw[dashed](H3) to node[above]{$\{1\}$}(H2k-1);
					\draw(H1) to node[left]{$S$}(H2);
					\draw(H1) to node[above]{$L$} (H3);
					\draw(H1) to node[right]{$K$} (H4);
					\draw(H2) to node[left]{$R$}(H3);
					\draw(H2) to node[below]{$T$}(H4);
					\draw(H3) to node[right]{$U$}(H4);
				\end{tikzpicture}
			\end{center}
		\end{figure}
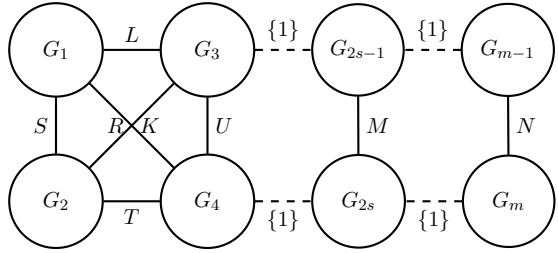
		
		According to condition~(c), there exists a $3$-cycle through $(g,i)$ for each $g\in G$ and $i\in\{1,2,3,4\}$.
		On the other hand, for each $i\in\{5,\dots,m\}$, there is no $3$-cycle through $(g,i)$ for any $g\in G$.
		Hence $A$ stabilizes $G_1\cup G_2\cup G_3\cup G_4$. Since $[G_1\cup G_2\cup G_3\cup G_4]\cong \Sigma$ is a $4$-PGSR of $G$,
		the group $A$ stabilizes $G_1$, $G_2$, $G_3$ and $G_4$ respectively, and $A_{1_1}$ fixes $G_1\cup G_2\cup G_3\cup G_4$ pointwise. Since $T_{3,5}=T_{4,6}=\{1\}$ and $T_{i,j}=\emptyset$ for pairs $(i,j)$ with $i \leq4$ and $j\geq 5$ other than $(3,5)$ and $(4,6)$,  it follows that $A$ stabilizes $G_5$ and $G_6$ respectively, and that $A_{1_1}$ fixes $G_5\cup G_6$ pointwise. Repeating this argument, we derive that $A$ stabilizes $G_i$ for each $i\in\{1,\dots,m\}$, and $A_{1_1}$ fixes $G_1\cup\dots\cup G_{m}$ pointwise.
		Therefore, $\Gamma_m$ is an $m$-HGR of $G$.
	\end{proof}
	
	\begin{lemma}\label{Lem5PGSR}
		Let $G$ be a finite group, and let $\Sigma$ be a $5$-PGSR of $G$ such that the following conditions are satisfied for some integer $k$ with $2\leq k\leq|G|$:
		\begin{enumerate}[\rm(a)]
			\item the vertices $(g,1)$, $(g,2)$ and $(g,3)$ have valency $k+1$ for each $g\in G$;
			\item the vertices $(g,4)$ and $(g,5)$ have valency $k$ for each $g\in G$;
			\item there exists a $3$-cycle through $(g,i)$ for each $g\in G$ and $i\in\{1,2,3,4,5\}$.
		\end{enumerate}
		Then $G$ has an $m$-HGR for each odd integer $m\geq 7$.
	\end{lemma}
	
	\begin{proof}
		Write $\Sigma=\Cay(G,(T_{i,j})_{5\times 5})$ with $T_{1,2}=S$, $T_{1,3}=L$, $T_{1,4}=K$, $T_{1,5}=R$, $T_{2,3}=T$, $T_{2,4}=U$,
		$T_{2,5}=Z$, $T_{3,4}=W$, $T_{3,5}=X$ and $T_{4,5}=Y$. Then by conditions~(a) and~(b), we have $|S|+|L|+|K|+|R|=|S|+|T|+|U|+|Z|=|L|+|T|+|W|+|X|=k+1$ and $|K|+|U|+|W|+|Y|=|R|+|Z|+|X|+|Y|=k$. Take any subsets $M$ and $N$ of $G$ of size $k-1$ and $k$, respectively. Let $\Gamma_m=\Cay(G,(T_{i,j})_{m\times m})$ with $T_{1,2}=S$, $T_{1,3}=L$, $T_{1,4}=K$, $T_{1,5}=R$, $T_{2,3}=T$, $T_{2,4}=U$,
		$T_{2,5}=Z$, $T_{3,4}=W$, $T_{3,5}=X$, $T_{4,5}=Y$, $T_{m-1,m}=N$ and
		\[
		T_{i,j}=
		\begin{cases}
			M&\text{if }j=i+1\text{ is odd and }6\leq i\leq m-3,\\
			\{1\}&\text{if }j=i+2\text{ and }4\leq i\leq m-2,\\
			\emptyset&\text{otherwise}.
		\end{cases}
		\]
		The structure of $\Gamma_m$ is illustrated in Figure~\ref{Fig11}. Since
		\begin{align*}
			k+1&=|S|+|L|+|K|+|R|\\
			&=|S|+|T|+|U|+|Z|\\
			&=|L|+|T|+|W|+|X|\\
			&=|K|+|U|+|W|+|Y|+1\\
			&=|R|+|Z|+|X|+|Y|+1\\
			&=|M|+2\\
			&=|N|+1,
		\end{align*}
		it follows that $\Gamma_m$ is $(k+1)$-regular and hence an $m$-Haar graph of $G$. Let $A=\Aut(\Gamma_m)$.
		
		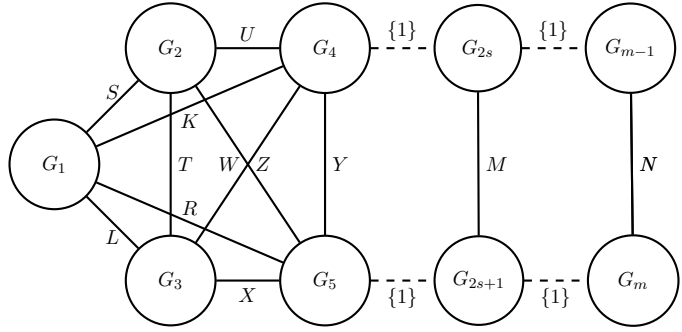
\begin{figure}[!hhh]
			\vspace{3mm}
			\caption{From $5$-PGSR to $m$-HGR for odd $m\geq 7$}\label{Fig11}
			\begin{center}
				\begin{tikzpicture}[node distance=1.2cm,thick,scale=0.7,every node/.style={transform shape}]
					\node[circle](C0){};
					\node[left=of C0,circle,draw,inner sep=11.5pt, label=45:{}](G2){$G_{2}$};
					\node[below=of G2](G21){};
					\node[below=of G21,circle,draw,inner sep=11.5pt,label=-90:{}](G3){$G_{3}$};
					\node[left=of G21,circle,draw,inner sep=11.5pt,label=-90:{}](G1){$G_1$};
					\node[right=of G2,circle,draw,inner sep=11.5pt,label=-90:{}](G4){$G_{4}$};
					\node[right=of G3,circle,draw,inner sep=11.5pt,label=-90:{}](G5){$G_{5}$};
					\node[right=of G5,circle,draw,inner sep=6.8pt,label=-90:{}](G2k+1){$G_{2s+1}$};
					\node[right=of G4,circle,draw,inner sep=9.8pt,label=-90:{}](G2k){$G_{2s}$};
					\node[right=of G2k+1,circle,draw,inner sep=10.8pt,label=-90:{}](Gm){$G_{m}$};
					\node[right=of G2k,circle,draw,inner sep=7.2pt,label=-90:{}](Gm-1){$G_{m-1}$};
					\draw(Gm) to node[right]{$N$} (Gm-1);
					\draw(G1) to node[above]{$S$}(G2);
					\draw(G1) to node[below]{$L$} (G3);
					\draw(G1) to node[below]{$K$}(G4);
					\draw(G1) to node[above]{$R$} (G5);
					\draw(G2) to node[right]{$T$}(G3);
					\draw(G2) to node[above]{$U$}(G4);
					\draw(G2) to node[right]{$Z$}(G5);
					\draw(G3) to node[left]{$W$}(G4);
					\draw(G3) to node[below]{$X$}(G5);
					\draw(G4) to node[right]{$Y$}(G5);
					\draw[dashed](G4) to node[above]{$\{1\}$}(G2k);
					\draw[dashed](G5) to node [below]{$\{1\}$}(G2k+1);
					\draw[dashed](G2k+1) to node[below]{$\{1\}$}(Gm);
					\draw[dashed](G2k) to node[above]{$\{1\}$}(Gm-1);
					\draw(G2k) to node[right]{$M$}(G2k+1);
					\draw(Gm-1) to node[right]{$N$}(Gm);
				\end{tikzpicture}
			\end{center}
		\end{figure}
		
		According to condition~(c), there exists a $3$-cycle through $(g,i)$ for each $g\in G$ and $i\in\{1,2,3,4,5\}$.
		On the other hand, for each $i\in\{6,\dots,m\}$, there is no $3$-cycle through $(g,i)$ for any $g\in G$.
		Hence $A$ stabilizes $G_1\cup G_2\cup G_3\cup G_4\cup G_5$. Since $[G_1\cup G_2\cup G_3\cup G_4\cup G_5]\cong \Sigma$ is a $5$-PGSR of $G$, the group $A$ stabilizes $G_1$, $G_2$, $G_3$, $G_4$ and $G_5$ respectively, and $A_{1_1}$ fixes $G_1\cup G_2\cup G_3\cup G_4\cup G_5$ pointwise. Since $T_{4,6}=T_{5,7}=\{1\}$ and $T_{i,j}=\emptyset$ for pairs $(i,j)$ with $i \leq5$ and $j\geq 6$ for other subscript pairs of $T$ other than $(4,6)$ and $(5,7)$,  it follows that $A$ stabilizes $G_6$ and $G_7$ respectively, and that $A_{1_1}$ fixes $G_6\cup G_7$ pointwise. Repeating this argument, we derive that $A$ stabilizes $G_i$ for each $i\in\{1,\dots,m\}$, and $A_{1_1}$ fixes $G_1\cup\dots\cup G_{m}$ pointwise. Therefore, $\Gamma_m$ is an $m$-HGR of $G$.
	\end{proof}

	\section{Some small groups}\label{Sec3}
	
	In this section we consider finite groups in the set
	\begin{equation}\label{Eq-3.0}
		\mathcal{G}_0=\{C_1,C_2,C_3,C_4,C_5,C_6,C_2^2,C_2^3,C_3^2,D_6,A_4,3_+^{1+2}\},
	\end{equation}
	where $3_+^{1+2}$ denotes the nonabelian group of order $27$ with exponent $3$.
	Note that this covers all the groups appearing in~(a)--(d) of Theorem~\ref{theo=main}.
	
	\begin{lemma}\label{Lem3.1}
		Let $m\geq3$ be an integer and let $G=C_n$ with $1\leq n\leq6$.
		Then $G$ has no $m$-HGR if and only if one of the following holds:
		\begin{enumerate}[\rm(a)]
			\item $m=3$, and $G=C_1, C_2$, $C_3$,  $C_4$ or $C_5$;
			\item $m=4$, and $G=C_1$, $C_2$ or $C_3$;
			\item $m=5$, and $G=C_1$ or $C_2$;
			\item $6\leq m\leq 9$, and $G=C_1$.
		\end{enumerate}
	\end{lemma}
	
	\begin{proof}
		First let $G=C_1$. Then a graph $\Gamma$ is an $m$-HGR of $G$ if and only if $\Gamma$ is a regular G$m$SR of $G$. By~\cite[Theorem 1.1]{DFS}, $G$ has a regular G$m$SR if and only if $m\geq 10$. Hence $G$ has an $m$-HGR if and only if $m\geq 10$.\smallskip
		
		Next let $G=\langle x\rangle=C_2$. Searching in \magma~\cite{magma} shows that $G$ has no $m$-HGRs for $m\in\{3,4,5\}$ and has $m$-HGRs $\Cay(G,(T_{i,j})_{m\times m})$ for $m\in\{6,\ldots,9\}$ as follows.
		\begin{itemize}
			\item $m=6$: $T_{1,2}=T_{1,6}=T_{3,5}=T_{4,6}=\{1,x\}$, $T_{1,5}=T_{2,3}=T_{2,5}=T_{3,6}=T_{4,5}=\{1\}$,  $T_{2,4}=T_{3,4}=\{x\}$ and $T_{i,j}=\emptyset$ otherwise;
			\item $m=7$: $T_{1,2}=\{1,x\}$, $T_{1,3}=T_{1,7}=T_{2,3}=T_{4,7}=T_{5,6}=T_{5,7}=T_{6,7}=\{x\}$, $T_{2,6}=T_{3,4}=T_{3,5}=T_{4,5}=T_{4,6}=\{1\}$ and $T_{i,j}=\emptyset$ otherwise;
			\item $m=8$: $T_{1,2}=\{1,x\}$, $T_{1,7}=T_{1,8}=T_{3,4}=T_{3,8}=T_{4,5}=T_{4,6}=T_{5,8}=\{1\}$, $T_{2,3}=T_{2,6}=T_{3,5}=T_{4,7}=T_{5,6}=T_{6,7}=T_{7,8}=\{x\}$ and $T_{i,j}=\emptyset$ otherwise;
			\item $m=9$: $T_{1,2}=T_{8,9}=\{1,x\}$, $T_{1,7}=T_{1,9}=T_{3,4}=T_{3,8}=T_{4,5}=T_{4,6}=T_{5,9}=\{1\}$, $T_{2,3}=T_{2,6}=T_{3,5}=T_{4,7}=T_{5,6}=T_{6,7}=T_{7,8}=\{x\}$ and $T_{i,j}=\emptyset$ otherwise.
		\end{itemize}
		For each $m\ge 10$, according to a result of Baron and Imrich~\cite{BI}, there exists a connected $4$-regular asymmetric graph $\Sigma_m$ of order $m$. Let $\Gamma_m$ be a disconnected union of two copies of $\Sigma_m$. Then $\Aut(\Gamma_m)=\Aut(\Sigma_m)\wr S_2\cong C_2$, and so $\Gamma_m$ is an $m$-HGR of $G$. This proves that $G$ has an $m$-HGR for each $m\geq 10$.
		
		Now let $G=\langle x\rangle =C_3$. Searching in \magma~\cite{magma} shows that $G$ has no $3$-HGR or $4$-HGR but has a $5$-HGR $\Cay(G,(T_{i,j})_{5\times5})$ with $T_{1,2}=\{1,x\}$,  $T_{1,3}=T_{4,5}=\{1,x^{-1}\}$, $T_{2,4}=\{1\}$, $T_{2,5}=T_{3,4}=T_{3,5}=\{x\}$ and $T_{i,j}=\emptyset$ otherwise. Let $\Sigma_4=\Cay(G,(T_{i,j})_{4\times4})$ with
		\[
		T_{1,2}=T_{1,3}=T_{1,4}=T_{2,3}=\{1,x\},\ \ T_{2,4}=\{x,x^{-1}\}, \ \  T_{3,4}=\{x\},
		\]
		and let $\Sigma_5=\Cay(G,(T_{i,j})_{5\times5})$ with
		\[
		T_{1,2}=T_{1,3}=T_{2,3}=T_{4,5}=\{1,x\},\ \ T_{1,4}=T_{1,5}=\{1\},\  \ T_{2,4}=T_{3,4}=\{x^{-1}\},\ \ T_{2,5}=T_{3,5}=\{x\}.
		\]
		Clearly, $\Sigma_4$ satisfies conditions (a)--(c) of Lemma~\ref{Lem4PGSR} and $\Sigma_5$ satisfies conditions (a)--(c) of Lemma~\ref{Lem5PGSR}. Moreover, it is verified by \magma~\cite{magma} that $\Sigma_4$ is a $4$-PGSR and $\Sigma_5$ is a $5$-PGSR of $G$. Then we derive from Lemmas~\ref{Lem4PGSR} and~\ref{Lem5PGSR} that $G$ has an $m$-HGR for each $m\geq6$.
		
		Finally, let $G=\langle x\rangle =C_n$ with $4\leq n\leq 6$. Searching in \magma~\cite{magma} shows that neither $C_4$ nor $C_5$ has a $3$-HGR, and that $G$ has the following $m$-HGRs $\Cay(G,(T_{i,j})_{m\times m})$ with $m\in\{3,4\}$.
		\begin{itemize}
			\item $m=3$ and $G=C_6$: $T_{1,2}=\{1,x^3\}$, $T_{1,3}=\{1,x^{-1}\}$, $T_{2,3}=\{x,x^{-1}\}$;
			\item $m=4$ and $G=C_n$ with $4\leq n\leq 6$: $T_{1,2}=T_{1,3}=\{1,x\}$, $T_{1,4}=\{1\}$, $T_{2,3}=\{x\}$, $T_{2,4}=T_{3,4}=\{x,x^{-1}\}$.
		\end{itemize}
		Let $\Sigma_3=\Cay(G,(T_{i,j})_{3\times3})$ with
		\[
		T_{1,2}=\{1,x\},\ \ T_{1,3}=\{x,x^{-1}\},\ \ T_{2,3}=\{1\},
		\]
		and let $\Sigma_4=\Cay(G,(T_{i,j})_{4\times4})$ with
		\[
		T_{1,2}=T_{1,3}=\{1,x\},\ \ T_{1,4}=\{1\},\ \ T_{2,3}=\{x\},\ \ T_{2,4}=\{x,x^2\},\ \ T_{3,4}=\{x^{-1}\}.
		\]
		It is clear that $\Sigma_3$ satisfies conditions (a)--(c) of Lemma~\ref{Lem3PGSR} and $\Sigma_4$ satisfies conditions (a)--(c) of Lemma~\ref{Lem4PGSR}. By \magma~\cite{magma} one may verify that $\Sigma_3$ is a $3$-PGSR and $\Sigma_4$ is a $4$-PGSR of $G$. Then it follows from Lemmas~\ref{Lem3PGSR} and~\ref{Lem4PGSR} that $G$ has an $m$-HGR for each $m\geq 5$.
	\end{proof}
	
	\begin{lemma}\label{Lem3.2}
		Let $m\geq 3$ be an integer and let $G=C_2^2$ or $C_2^3$. Then $G$ has no $m$-HGR if and only if $m=3$ and $G=C_2^2$.
	\end{lemma}
	
	\begin{proof}
		First let $G=\langle x,y\mid x^2=y^2=(xy)^2=1\rangle=C_2^2$. Search in \magma~\cite{magma} shows that $G$ has no $3$-HGR but has a $4$-HGR $\Cay(G,(T_{i,j})_{4\times4})$ with
		\[
		T_{1,2}=T_{2,3}=T_{3,4}=\{1,x\},\ \ T_{1,3}=\{x\},\ \ T_{1,4}=\{x,y\},\ \ T_{2,4}=\{y\}
		\]
		and a $5$-HGR $\Cay(G,(T_{i,j})_{5\times5})$ with
		\[
		T_{1,2}=\{1,x\},\ \ T_{1,3}=\{x,y\},\  \ T_{2,3}=T_{3,5}=\{1\},\ \ T_{2,4}=\{y\}, \ \ T_{4,5}=\{1,x,y\}
		\]
		and $T_{i,j}=\emptyset$ otherwise. Let $\Sigma_4=\Cay(G,(T_{i,j})_{4\times4})$ with
		\[
		T_{1,2}=T_{2,3}=\{1,x\},\ \ T_{1,3}=T_{3,4}=\{x\},\ \ T_{1,4}=\{x,y\},\ \ T_{2,4}=\{y\},
		\]
		and let $\Sigma_5=\Cay(G,(T_{i,j})_{5\times5})$ with
		\[
		T_{1,2}=T_{1,3}=T_{4,5}=\{1,x,y\},\ \ T_{2,3}=T_{2,4}=\{x\},\  \ T_{2,5}=\{xy\},\ \ T_{3,4}=T_{3,5}=\{y\},
		\ \ T_{1,4}=T_{1,5}=\emptyset.
		\]
		Clearly, $\Sigma_4$ satisfies conditions (a)--(c) of Lemma~\ref{Lem4PGSR} and $\Sigma_5$ satisfies conditions (a)--(c) of Lemma~\ref{Lem5PGSR}. Moreover, computation in \magma~\cite{magma} shows that $\Sigma_4$ is a $4$-PGSR and $\Sigma_5$ is a $5$-PGSR of $G$. Then it follows from Lemmas~\ref{Lem4PGSR} and~\ref{Lem5PGSR} that $G$ has an $m$-HGR for each $m\geq6$.
		
		Next let $G=\langle x,y,z~|~x^2=y^2=z^2=(xy)^2=(xz)^2=(yz)^2=1\rangle= C_2^3$.
		Computation in \magma~\cite{magma} verifies that $\Cay(G,(T_{i,j})_{3\times3})$ with
		\[
		T_{1,2}=\{1,x,z,xy\},\ \ T_{1,3}=\{z,xy,xz,xyz\},\ \ T_{2,3}=\{y,z,xy,xz\}
		\]
		is a $3$-HGR of $G$, and $\Cay(G,(T_{i,j})_{4\times4})$ with
		\[
		T_{1,2}=\{1,x\},\ \ T_{1,3}=\{x,z\},\ \ T_{1,4}=T_{2,3}=\{x\}, \ \ T_{2,4}=\{x,y\},\ \ T_{3,4}=\{x,z\}
		\]
		is a $4$-HGR of $G$. Let $\Sigma_3=\Cay(G,(T_{i,j})_{3\times3})$ with
		\[
		T_{1,2}=\{1,x,y\},\ \ T_{1,3}=\{1,xz,xyz\},\ \ T_{2,3}=\{xz,yz\},
		\]
		and let $\Sigma_4=\Cay(G,(T_{i,j})_{4\times4})$ with
		\[
		T_{1,2}=\{1,x\},\ \ T_{1,3}=\{x,z\},\ \ T_{1,4}=T_{2,3}=\{x\}, \ \ T_{2,4}=\{x,y\},\ \ T_{3,4}=\{y\}.
		\]
		Clearly, $\Sigma_3$ satisfies conditions (a)--(c) of Lemma~\ref{Lem3PGSR} and $\Sigma_4$ satisfies conditions (a)--(c) of Lemma~\ref{Lem4PGSR}. Moreover, it is verified in \magma~\cite{magma} that $\Sigma_3$ is a $3$-PGSR and $\Sigma_4$ is a $4$-PGSR of $G$. Hence Lemmas~\ref{Lem3PGSR} and~\ref{Lem4PGSR} imply that $G$ has an $m$-HGR for each $m\geq 5$.
	\end{proof}
	
	\begin{lemma}\label{Lem3.3}
		Let $m\geq 3$ be an integer and let $G\in\{C_3^2,D_6,A_4,3_+^{1+2}\}$. Then $G$ has no $m$-HGR if and only if $m=3$ and $G=D_6$.
	\end{lemma}
	
	\begin{proof}
		Let $x$ and $y$ be elements of $G$ such that $G=\langle x,y\mid x^3=y^3=1,\,xy=yx\rangle$ if $G=C_3^2$, $G=\langle x,y\mid x^3=y^2=(xy)^2=1\rangle$ if $G=D_6$,  $G=\langle x,y\mid x^3=y^2=(xy)^3=1\rangle$ if $G=A_4$, and $G=\langle x,y,z\mid x^3=y^3=z^3=[x,z]=[y,z]=1,\, [x,y]=z\rangle$ if $G=3_+^{1+2}$. Searching in \magma~\cite{magma} shows that the $D_6$ has no 3-HGR, and $G$ has $m$-HGRs $\Cay(G,(T_{i,j})_{m\times m})$ for $m\in\{3,4\}$ as follows.
		\begin{itemize}
			\item $m=3$ and $G=C_3^2$, $A_4$ or $3_+^{1+2}$: $T_{1,2}=\{1,x,y\}$, $T_{1,3}=\{1,x,xy\}$, $T_{2,3}=\{1,x^{-1},yx\}$;
			\item $m=4$ and $G=C_3^2$, $D_6$, $A_4$ or $3_+^{1+2}$: $T_{1,2}=\{1,y\}$, $T_{1,3}=T_{2,4}=\{1,x\}$, $T_{1,4}=T_{2,3}=\{1\}$, $T_{3,4}=\{x,y\}$.
		\end{itemize}
		Let $\Sigma_3=\Cay(G,(T_{i,j})_{3\times3})$ with
		\[
		T_{1,2}=\{1,x,y\},\ \ T_{1,3}=\{1,x,xy\},\ \ T_{2,3}=\{1,yx\},
		\]
		and let $\Sigma_4=\Cay(G,(T_{i,j})_{4\times4})$ with
		\[
		T_{1,2}=\{1,y\},\ \ T_{1,3}=T_{2,4}=\{1,x\},\ \ T_{1,4}=T_{2,3}=\{1\},\ \  T_{3,4}=\{x\}.
		\]
		It is evident that $\Sigma_3$ satisfies conditions (a)--(c) of Lemma~\ref{Lem3PGSR} and $\Sigma_4$ satisfies conditions (a)--(c) of Lemma~\ref{Lem4PGSR}. By \magma~\cite{magma} one may verify that $\Sigma_3$ is a $3$-PGSR and $\Sigma_4$ is a $4$-PGSR of $G$. Then it follows from Lemmas~\ref{Lem3PGSR} and~\ref{Lem4PGSR} that $G$ has an $m$-HGR for each $m\geq 5$.
	\end{proof}
	
	A combination of Lemmas~\ref{Lem3.1},~\ref{Lem3.2} and~\ref{Lem3.3} leads to the following result.
	
	\begin{proposition}\label{prop=smallgroups}
		Let $m\geq 3$ be an integer and let $G\in\mathcal{G}_0$.
		Then $G$ has no $m$-HGR if and only if one of the following holds:
		\begin{enumerate}[\rm(a)]
			\item $m=3$, and $G=C_1, C_2$, $C_3$,  $C_4$, $C_5$, $C_2^2$ or $D_6$.
			\item $m=4$, and $G=C_1$, $C_2$ or $C_3$.
			\item $m=5$, and $G=C_1$ or $C_2$.
			\item $6\leq m\leq 9$, and $G=C_1$.
		\end{enumerate}
	\end{proposition}

	\section{Groups with a small generating set}\label{Sec4}
	
	This section is devoted to the groups $G$ that has a generating set of at most three elements but is not in $\mathcal{G}_0$ (see~\eqref{Eq-3.0} for the definition of $\mathcal{G}_0$).

\begin{lemma}\label{Lem=5.0}
	Let $G$ be a finite group with $d(G)=2$. Then either $G$ has a generating set $\{x,y\}$ with $|x|\geq4$, or $G\in\{C_2^2,C_3^2,D_6,A_4,3_+^{1+2}\}$.
\end{lemma}

\begin{proof}
	Suppose that $G$ has no a generating set $\{x,y\}$ with $|x|\geq4$. Take an arbitrary generating set $\{a,b\}$ of size two in $G$. Then since $G=\langle a,b\rangle=\langle a,ab\rangle=\langle ab,b\rangle$, it follows that $|a|,|b|,|ab|\in\{2,3\}$. If at least one of $|a|$, $|b|$ or $|ab|$ is $2$, then it is easy to see that $G\in\{C_2^2,D_6,A_4\}$. Now assume that $|a|=|b|=|ab|=3$. Then since $G=\langle a,a^2b\rangle$, we have $|a^2b|\in\{2,3\}$. If $|a^2b|=2$, then $G=A_4$. If $|a^2b|=3$, then $G=C_3^2$ or $3_+^{1+2}$.
\end{proof}

The following lemma will be applied several
times in this section.
	
		\begin{lemma}\label{Lem=stabilizertrivial}
		Let $m\geq 2$ be an integer, let $G$ be a finite group, and let $\Gamma=\Cay(G,(T_{i,j})_{m\times m})$ be an $m$-Cayley graph of $G$
		such that $\left\langle   T_{i,j}\mid i,j\in \{1,\dots,m\}\right\rangle =G$. Suppose that $A:=\Aut(\Gamma)$  stabilizes $G_1$ and $A_{1_1}=A_{1_i}$ fixes $\Gamma(1_i)$ pointwise for each $i\in \{1,\dots,m\}$. Then $A=G$.
	\end{lemma}
	
	\begin{proof}
			For $i,j\in \{1,\dots,m\}$, since $A_{1_i}$ fixes $\Gamma(1_i)$ pointwise, we have $A_{1_i}\leq A_{x_j}$ for each $x\in T_{i,j}$ and so $A_{g_i}\leq A_{(xg)_j}$ for each $g\in G$ and $x\in T_{i,j}$.  From $A_{1_1}=\cdots=A_{1_m}$, we obtain $A_{g_1}=\cdots=A_{g_m}$ for each $g\in G$. Then since $\left\langle  T_{i,j}\mid i,j\in\{1,\dots,m\} \right\rangle =G$, it follows that, for each $g\in G$, the group $A_{1_1}$ fixes $g_i$ for some $i\in\{1,\dots,m\}$, whence $A_{1_1}\leq A_{g_i}=A_{g_j}$ for all $j\in \{1,\dots,m\}$. This shows that $A_{1_1}$ fixes every vertex of $\Gamma$, and so $A_{1_1}=1$. Therefore, $A=G$ as $A$ stabilizes $G_1$.	
		\end{proof}
	
	We deal with groups $G$ with $d(G)\leq2$ and $d(G)=3$, respectively, in the following two lemmas.
	
	\begin{lemma}\label{lemma=d(G)=2}
		Let $G$ be a group with $d(G)\leq2$ such that $G\notin\mathcal{G}_0$. Then $G$ has an $m$-HGR for each $m\geq 3$.
	\end{lemma}
	
	\begin{proof}
		First of all, let us take two elements $x$ and $y$ in $G$ as follows. If $d(G)=1$, then $G=\langle a\rangle$ for some element $a$ of order at least $7$ (since $G\notin\mathcal{G}_0$), and in this case we take $x=a$ and $y=a^3$. If $d(G)=2$, then by Lemma~\ref{Lem=5.0}, there are elements $x$ and $y$ of $G$ such that $G=\langle x,y\rangle$ and $|x|\geq4$.
		In the following we divide the proof into four cases according to the value of $m$.
		
		\smallskip
		\textsf{Case 1}: $m=3$.
		\smallskip
		
		Let $\Gamma=\Cay(G,(T_{i,j})_{3\times3})$ with $T_{1,2}=\{1,x,y^{-1}\}$, $T_{1,3}=\{1,x,x^{-1}\}$ and $T_{2,3}=\{x,x^{-1},y\}$, and let $A=\Aut(\Gamma)$. Then $\Gamma$ is a $3$-Haar graph of $G$ with valency $6$ such that
		\begin{align*}
			\Gamma(1_1)&=\{1_2,x_2,(y^{-1})_2,1_3,x_3,(x^{-1})_3\},\\
			\Gamma(1_2)&=\{1_1,(x^{-1})_1,y_1,x_3,(x^{-1})_3,y_3\},\\
			\Gamma(1_3)&=\{1_1,(x^{-1})_1,x_1,(x^{-1})_2,x_2,(y^{-1})_2\}.
		\end{align*}
		With $[\Gamma(1_1)]$, $[\Gamma(1_2)]$ and $[\Gamma(1_3)]$ depicted in Figure~\ref{Fig6}, one sees readily that $[\Gamma(1_1)]\ncong [\Gamma(1_i)]$ for $i\in\{2,3\}$. Consequently, $A$ stabilizes $G_1$.
		With this in mind, it is evident from Figure~\ref{Fig6} that $A_{1_i}$ fixes $\Gamma(1_i)$ pointwise for each $i\in \{2,3\}$.
		
		\begin{figure}[!ht]
			\vspace{3mm}
			\caption{The induced subgraph $[\Gamma(1_i)]$ with $i\in\{1,2,3\}$}\label{Fig6}
			\begin{tikzpicture}[node distance=0.9cm,thick,scale=0.63,every node/.style={transform shape},scale=1.1]
				\node[](A0){};
				\node[right=of A0](D00){};
				\node[right=of D00](D000){};
				\node[right=of D000](D0000){};
				\node[right=of D0000](D00000){};
				\node[right=of D00000](D000000){};
				\node[right=of D0000](Dp){};
				\node[above=of Dp, circle,draw, inner sep=1pt, label=left:{\Large$(y^{-1})_2$}](D0){};
				\node[above=of D0, circle,draw, inner sep=1pt, label=left:{\Large$1_3$}](D1){};
				\node[above=of D1, circle,draw, inner sep=1pt, label=left:{\Large$x_2$}](D2){};
				\node[below=of D0, circle,draw, inner sep=1pt, label=left:{\Large$x_3$}](D3){};
				\node[below=of D3, circle,draw, inner sep=1pt, label=left:{\Large$1_2$}](D4){};
				\node[below=of D4, circle,draw, inner sep=1pt, label=left:{\Large$(x^{-1})_3$}](D5){};
				\node[below=of D5](D50){};
				\node[right=of D50, label=left:{\Large$[\Gamma(1_1)]$}](){};
				\draw[] (D0) to (D1);
				\draw[] (D2) to (D1);
				\draw[] (D3) to (D4);
				\draw[] (D5) to (D4);
				\node[right=of Dp](X0){};
				\node[right=of X0](X01){};
				\node[right=of X01](X02){};
				\node[right=of X02](X00){};
				\node[right=of X00](X2){};\node[right=of X2](X3){};
				\node[above=of X2, circle,draw, inner sep=1pt, label=left:{\Large$x_3$}](W0){};
				\node[above=of W0, circle,draw, inner sep=1pt, label=left:{\Large$y_3$}](W1){};
				\node[above=of W1, circle,draw, inner sep=1pt, label=left:{\Large$y_1$}](W2){};
				\node[below=of W0, circle,draw, inner sep=1pt, label=left:{\Large$1_1$}](W3){};
				\node[below=of W3, circle,draw, inner sep=1pt, label=left:{\Large$(x^{-1})_3$}](W4){};
				\node[below=of W4, circle,draw, inner sep=1pt, label=left:{\Large$(x^{-1})_1$}](W5){};
				\node[below=of W5](W50){};
				\node[right=of W50, label=left:{\Large$[\Gamma(1_2)]$}](){};
				\draw[] (W1) to (W2);
				\draw[] (W0) to (W3);
				\draw[] (W3) to (W4);
				\draw[] (W5) to (W4);
				\node[right=of X2](Y00){};
				\node[right=of Y00](Y01){};
				\node[right=of Y01](Y02){};
				\node[right=of Y02](Y20){};
				\node[right=of Y20](E1){};
				\node[right=of E1](E2){};
				\node[right=of E2](E3){};
				\node[above=of E1, circle,draw, inner sep=1pt, label=left:{\Large$(y^{-1})_2$}](F0){};
				\node[above=of F0, circle,draw, inner sep=1pt, label=left:{\Large$(x^{-1})_2$}](F1){};
				\node[above=of F1, circle,draw, inner sep=1pt, label=left:{\Large$(x^{-1})_1$}](F2){};
				\node[below=of F0, circle,draw, inner sep=1pt, label=left:{\Large$1_1$}](F3){};
				\node[below=of F3, circle,draw, inner sep=1pt, label=left:{\Large$x_2$}](F4){};
				\node[below=of F4, circle,draw, inner sep=1pt, label=left:{\Large$x_1$}](F5){};
				\node[below=of F5](F50){};
				\node[right=of F50, label=left:{\Large$[\Gamma(1_3)]$}](){};
				\draw[] (F1) to (F2);
				\draw[] (F3) to (F4);
				\draw[] (F3) to (F0);
				\draw[] (F4) to (F5);
			\end{tikzpicture}
			\vspace{3mm}
		\end{figure}
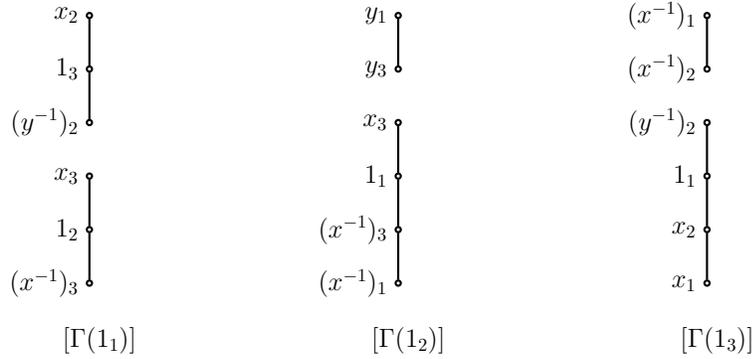
		
		Suppose that there exists $\a\in A$ with $(1_2)^{\a}=1_3$. Then $\a$ maps $[\Gamma(1_2)]$ to $[\Gamma(1_3)]$. Since $\alpha$ stabilizes $G_1$, it maps the $3$-path $(x_3,1_1,(x^{-1})_3,(x^{-1})_1)$ to the $3$-path $((y^{-1})_2,1_1,x_2,x_1)$.
		In particular, $(1_1)^{\a}=1_1$ and $((x^{-1})_1)^{\a}=x_1$.
		Then $\a$ stabilizes $\Gamma({1_1})$, which together with $(1_2)^{\a}=1_3$ implies that $(1_3)^{\a}=1_2$.
		This in turn indicates that $\a$ maps $[\Gamma(1_3)]$ to $[\Gamma(1_2)]$, and so it maps the arc $((x^{-1})_2,(x^{-1})_1)$ to the arc $(y_3,y_1)$. However, this leads to $((x^{-1})_1)^\alpha=y_1$, which contradicts the conclusion $((x^{-1})_1)^{\a}=x_1$.
		
		Thus, $A$ stabilizes $G_2$ and $G_3$ respectively. It follows that $A_{1_1}$ fixes $\Gamma(1_1)$ pointwise. Now $A_{1_i}$ fixes $\Gamma(1_i)$ pointwise for each $i\in\{1,2,3\}$. In particular, $A_{1_1}=A_{1_2}=A_{1_3}$. Then we conclude that $\Gamma$ is a $3$-HGR of $G$ by Lemma~\ref{Lem=stabilizertrivial}.
		
		\smallskip
		\textsf{Case 2}: $m=4$.
		\smallskip
		
		Let $\Gamma=\Cay(G,(T_{i,j})_{4\times4})$ with $T_{1,2}=T_{1,3}=T_{2,4}=\{1,x\}$, $T_{1,4}=T_{2,3}=\{1\}$ and $T_{3,4}=\{x,y\}$, and let $A=\Aut(\Gamma)$. Then $\Gamma$ is a $4$-Haar graph of $G$ with valency $5$ such that
		\begin{align*}
			\Gamma(1_1)&=\{1_2,x_2,1_3,x_3,1_4\},\\
			\Gamma(1_2)&=\{1_1,(x^{-1})_1,1_3,1_4,x_4\},\\
			\Gamma(1_3)&=\{1_1,(x^{-1})_1,1_2,x_4,y_4\},\\
			\Gamma(1_4)&=\{1_1,1_2,(x^{-1})_2,(x^{-1})_3,(y^{-1})_3\}.
		\end{align*}
		The induced subgraphs $[\Gamma(1_1)]$, $[\Gamma(1_2)]$, $[\Gamma(1_3)]$ and $[\Gamma(1_4)]$ are depicted in Figure~\ref{Fig7}, which shows that $[\Gamma(1_i)]\ncong [\Gamma(1_j)]$ with distinct $i$ and $j$ in $\{1,\ldots,4\}$. Hence $A$ stabilizes $G_i$ for each $i\in\{1,\ldots,4\}$. Then it is clear from Figure~\ref{Fig7} that $A_{1_i}$ fixes $\Gamma(1_i)$ pointwise for $i\in \{1,2,4\}$, and $A_{1_3}$ fixes each of $1_2$, $x_4$ and $y_4$. This in turn implies that $A_{1_3}\leq A_{1_2}$ fixes $(x^{-1})_1$ and thus fixes $\Gamma(1_3)$ pointwise. Now that $A_{1_i}$ fixes $\Gamma(1_i)$ pointwise for $i\in \{1,2,3,4\}$. In particular, $A_{1_1}=A_{1_2}=A_{1_3}=A_{1_4}$. Therefore, $\Gamma$ is a $4$-HGR of $G$ by Lemma~\ref{Lem=stabilizertrivial}.
		
		\begin{figure}[!ht]
			\vspace{3mm}
			\caption{The induced subgraph $[\Gamma(1_i)]$ with $i\in\{1,2,3,4\}$}\label{Fig7}
			\begin{tikzpicture}[node distance=0.9cm,thick,scale=0.63,every node/.style={transform shape},scale=1.1]
				\node[](A0){};
				\node[right=of A0](D00){};
				\node[right=of D00](D000){};
				\node[right=of D000](Dp){};
				\node[above=of Dp, circle,draw, inner sep=1pt, label=left:{\Large$x_3$}](D0){};
				\node[above=of D0, circle,draw, inner sep=1pt, label=left:{\Large$x_2$}](D1){};
				\node[below=of D0, circle,draw, inner sep=1pt, label=left:{\Large$1_3$}](D3){};
				\node[below=of D3, circle,draw, inner sep=1pt, label=left:{\Large$1_2$}](D4){};
				\node[below=of D4, circle,draw, inner sep=1pt, label=left:{\Large$1_4$}](D5){};
				\node[below=of D5](D50){};
				\node[right=of D50, label=left:{\Large$[\Gamma(1_1)]$}](){};
				\draw[] (D0) to (D1);
				\draw[] (D3) to (D4);
				\draw[] (D5) to (D4);
				\node[right=of Dp](X0){};
				\node[right=of X0](X01){};
				\node[right=of X01](X02){};
				\node[right=of X02](X00){};
				\node[above=of X00, circle,draw, inner sep=1pt, label=left:{\Large$1_1$}](W0){};
				\node[above=of W0, circle,draw, inner sep=1pt, label=left:{\Large$1_4$}](W1){};
				\node[below=of W0, circle,draw, inner sep=1pt, label=left:{\Large$1_3$}](W3){};
				\node[below=of W3, circle,draw, inner sep=1pt, label=left:{\Large$x_4$}](W4){};
				\node[below=of W4, circle,draw, inner sep=1pt, label=left:{\Large$(x^{-1})_1$}](W5){};
				\node[below=of W5](W50){};
				\node[right=of W50, label=left:{\Large$[\Gamma(1_2)]$}](){};
				\draw[] (W0) to (W1);
				\draw[] (W1) to (W0);
				\draw[] (W3) to (W0);
				\draw[] (W4) to (W3);
				\draw [bend left] (W3) to (W5);
				\node[right=of X00](Y0){};
				\node[right=of Y0](Y01){};
				\node[right=of Y01](Y02){};
				\node[right=of Y02](Y00){};
				\node[above=of Y00, circle,draw, inner sep=1pt, label=left:{\Large$(x^{-1})_1$}](F0){};
				\node[above=of F0, circle,draw, inner sep=1pt, label=left:{\Large$y_4$}](F1){};
				\node[below=of F0, circle,draw, inner sep=1pt, label=left:{\Large$1_2$}](F3){};
				\node[below=of F3, circle,draw, inner sep=1pt, label=left:{\Large$x_4$}](F4){};
				\node[below=of F4, circle,draw, inner sep=1pt, label=left:{\Large$1_1$}](F5){};
				\node[below=of F5](F50){};
				\node[right=of F50, label=left:{\Large$[\Gamma(1_3)]$}](){};
				\draw[] (F0) to (F3);
				\draw[] (F3) to (F4);
				\draw [bend left] (F3) to (F5);
				\node[right=of Y00](E0){};
				\node[right=of E0](E01){};
				\node[right=of E01](E02){};
				\node[right=of E02](E00){};
				\node[above=of E00, circle,draw, inner sep=1pt, label=left:{\Large$1_1$}](B0){};
				\node[above=of B0, circle,draw, inner sep=1pt, label=left:{\Large$(y^{-1})_3$}](B1){};
				\node[below=of B0, circle,draw, inner sep=1pt, label=left:{\Large$1_2$}](B3){};
				\node[below=of B3, circle,draw, inner sep=1pt, label=left:{\Large$(x^{-1})_2$}](B4){};
				\node[below=of B4, circle,draw, inner sep=1pt, label=left:{\Large$(x^{-1})_3$}](B5){};
				\node[below=of B5](B50){};
				\node[right=of B50, label=left:{\Large$[\Gamma(1_4)]$}](){};
				\draw[] (B0) to (B3);
				\draw[] (B5) to (B4);
			\end{tikzpicture}
			\vspace{3mm}
		\end{figure}
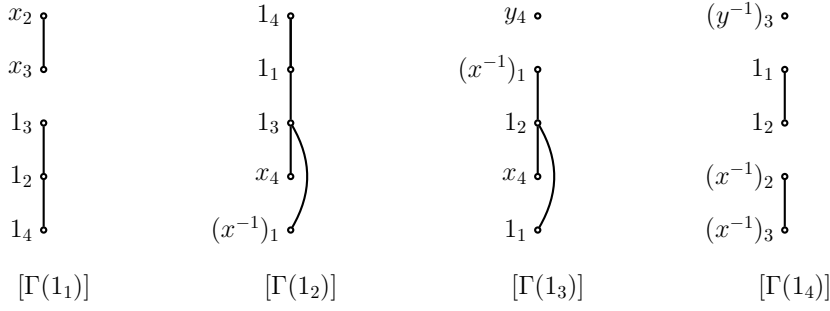
		
		\smallskip
		\textsf{Case 3}: $m\geq 5$ odd.
		\smallskip
		
		Let $\Sigma=\Cay(G,(T_{i,j})_{3\times3})$ with $T_{1,2}=\{1,x,y^{-1}\}$, $T_{1,3}=\{1,x,x^{-1}\}$ and $T_{2,3}=\{x,x^{-1}\}$.
	Clearly, $\Sigma$ satisfies conditions~(a)--(c) of Lemma~\ref{Lem3PGSR}, and
		\begin{align*}
			\Sigma(1_1)&=\{1_2,x_2,(y^{-1})_2,1_3,x_3,(x^{-1})_3\},\\
			\Sigma(1_2)&=\{1_1,(x^{-1})_1,y_1,x_3,(x^{-1})_3\},\\
			\Sigma(1_3)&=\{1_1,(x^{-1})_1,x_1,(x^{-1})_2,x_2\}.
		\end{align*}
		It is straightforward to verify the following statements by the definition of $\Sigma$:
		\begin{itemize}
			\item $[\Sigma(1_1)]$ is a union of an isolated vertex $(y^{-1})_2$, an edge $\{x_2,1_3\}$ and a $2$-path $(x_3,1_2,(x^{-1})_3)$;
			\item $[\Sigma(1_2)]$ is a union of an isolated vertex $y_1$ and a $3$-path $((x^{-1})_1,(x^{-1})_3,1_1,x_3)$;
			\item $[\Sigma(1_3)]$ is a union of an edge $\{(x^{-1})_1,(x^{-1})_2\}$ and a $2$-path $(1_1,x_2,x_1)$.
		\end{itemize}
		Let $A=\Aut(\Sigma)$. Then we conclude from the above statements that $A$ stabilizes $G_i$ for each $i\in\{1,2,3\}$, and then $A_{1_2}$ fixes $\Sigma(1_2)$ pointwise.  Moreover, $A_{1_1}$ fixes $1_2$, and $A_{1_3}$ fixes $x_2$. It follows that $A_{1_1}\leq A_{1_2}$ fixes every element of $\Sigma(1_2)$ including $x_3$, and $A_{1_3}\leq A_{x_2}$ fixes every element of $\Sigma(x_2)$ including $1_1$. Thus we infer that $A_{1_i}$ fixes $\Sigma(1_i)$ pointwise for $i\in\{1,3\}$. In particular, $A_{1_1}=A_{1_2}=A_{1_3}$. Therefore,  $\Sigma$ is a $3$-PGSR of $G$ by Lemma~\ref{Lem=stabilizertrivial}, and we are now in a position to apply Lemma~\ref{Lem3PGSR}, which asserts that $G$ has an $m$-HGR.
		
		\smallskip
		\textsf{Case 4}: $m\geq 6$ even.
		\smallskip
		
		Let $\Sigma=\Cay(G,(T_{i,j})_{4\times4})$ with $T_{1,2}=T_{1,3}=T_{2,4}=\{1,x\}$, $T_{1,4}=T_{2,3}=\{1\}$ and $T_{3,4}=\{y\}$.
		It is straightforward to verify that $\Sigma$ satisfies conditions~(a)--(c) of Lemma~\ref{Lem4PGSR} with
		\begin{align*}
			\Sigma(1_1)&=\{1_2,x_2,1_3,x_3,1_4\},\\
			\Sigma(1_2)&=\{1_1,(x^{-1})_1,1_3,1_4,x_4\},\\
			\Sigma(1_3)&=\{1_1,(x^{-1})_1,1_2,y_4\},\\
			\Sigma(1_4)&=\{1_1,1_2,(x^{-1})_2,(y^{-1})_3\},
		\end{align*}
		and the following statements hold:
		\begin{itemize}
			\item $[\Sigma(1_1)]$ is a union of an edge $\{x_2,x_3\}$ and a 2-path $(1_3,1_2,1_4)$;
			\item $[\Sigma(1_2)]$ is a union of an isolated vertex $x_4$ and a 3-path $(1_4,1_1,1_3,(x^{-1})_1)$;
			\item $[\Sigma(1_3)]$ is a union of an isolated vertex $y_4$ and a 2-path $(1_1,1_2,(x^{-1})_1)$;
			\item $[\Sigma(1_4)]$ is a union of two isolated vertices $(x^{-1})_2$ and $(y^{-1})_3$ and an edge $\{1_1,1_2\}$.
		\end{itemize}
		Let $A=\Aut(\Sigma)$. From the above statements we conclude that  $A$ stabilizes $G_i$ for each $i\in\{1,2,3,4\}$, and then $A_{1_i}$ fixes $\Sigma(1_i)$ pointwise for $i\in \{1,2,4\}$. Since $A_{1_3}$ fixes $1_2\in\Sigma(1_3)$, it follows that $A_{1_3}$ fixes $1_1\in\Sigma(1_2)$ and hence fixes $\Sigma(1_3)$ pointwise.
		Moreover, $A_{1_1}=A_{1_2}=A_{1_3}=A_{1_4}$.
	 Therefore, $\Sigma$ is a $4$-PGSR of $G$ by Lemma~\ref{Lem=stabilizertrivial}, and so Lemma~\ref{Lem4PGSR} asserts that $G$ has an $m$-HGR.
	\end{proof}
	
	\begin{lemma}\label{lemma=d(G)=3}
		Let $G$ be a group with $d(G)=3$ such that $G\notin\mathcal{G}_0$. Then $G$ has an $m$-HGR for each $m\geq 3$.
	\end{lemma}
	
	\begin{proof}
		Let $\{a,b,c\}$ be a generating set of $G$. If $|a|=|b|=|c|=2$, then at least one of $ab$, $ac$ or $bc$ has order at least $3$, for otherwise $G=C_2^3\in\mathcal{G}_0$, a contradiction. Then as $G=\langle a,ab,c\rangle=\langle a,b,ac\rangle=\langle a,bc,c\rangle$, there always exists a generating set $\{x,y,z\}$ of $G$ such that $|x|\geq3$.
		
		\smallskip
		\textsf{Case 1}: $m=3$.
		\smallskip
		
		Let $\Gamma=\Cay(G,(T_{i,j})_{3\times3})$ with $T_{1,2}=\{1,x,x^{-1}\}$, $T_{1,3}=\{1,y^{-1},z\}$ and $T_{2,3}=\{1,x^{-1},z\}$, and $A=\Aut(\Gamma)$. Then $\Gamma$ is a $3$-Haar graph of $G$ with valency $6$ such that
		\begin{align*}
			\Gamma(1_1)&=\{1_2,x_2,(x^{-1})_2,1_3,(y^{-1})_3,z_3\},\\
			\Gamma(1_2)&=\{1_1,(x^{-1})_1,x_1,1_3,(x^{-1})_3,z_3\},\\
			\Gamma(1_3)&=\{1_1,y_1,(z^{-1})_1,1_2,x_2,(z^{-1})_2\}.
		\end{align*}
		The induced subgraphs $[\Gamma(1_1)]$, $[\Gamma(1_2)]$ and $[\Gamma(1_3)]$ are depicted in Figure~\ref{Fig8}, which shows that $[\Gamma(1_1)]\ncong [\Gamma(1_i)]$ for $i\in\{2,3\}$. As a consequence, $A$ stabilizes $G_1$.
		It follows that $A_{1_2}$ fixes each of $1_1$, $x_1$, $(x^{-1})_1$ and $(x^{-1})_3$. Hence $A_{1_2}=A_{1_1}\cap A_{1_2}$ fixes both
		$1_3$ and $z_3$, and so $A_{1_2}$ fixes $\Gamma(1_2)$ pointwise. Similarly, $A_{1_3}$ fixes $\Gamma(1_3)$ pointwise.
		
		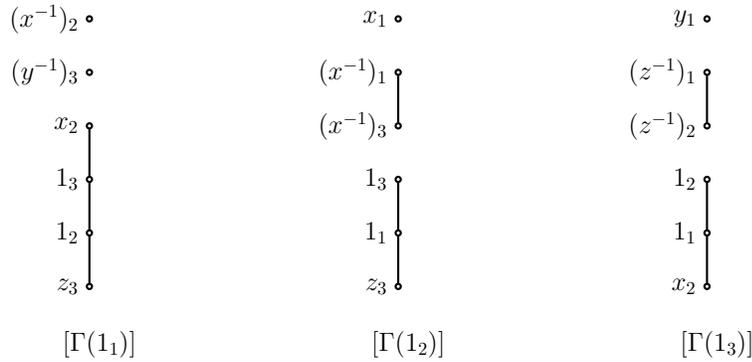
\begin{figure}[!ht]
			\vspace{3mm}
			\caption{The induced subgraph $[\Gamma(1_i)]$ with $i\in\{1,2,3\}$}\label{Fig8}
			\begin{tikzpicture}[node distance=0.9cm,thick,scale=0.63,every node/.style={transform shape},scale=1.1]
				\node[](A0){};
				\node[right=of A0](D00){};
				\node[right=of D00](D000){};
				\node[right=of D000](D0000){};
				\node[right=of D0000](D00000){};
				\node[right=of D00000](D000000){};
				\node[right=of D0000](Dp){};
				\node[above=of Dp, circle,draw, inner sep=1pt, label=left:{\Large$x_2$}](D0){};
				\node[above=of D0, circle,draw, inner sep=1pt, label=left:{\Large$(y^{-1})_3$}](D1){};
				\node[above=of D1, circle,draw, inner sep=1pt, label=left:{\Large$(x^{-1})_2$}](D2){};
				\node[below=of D0, circle,draw, inner sep=1pt, label=left:{\Large$1_3$}](D3){};
				\node[below=of D3, circle,draw, inner sep=1pt, label=left:{\Large$1_2$}](D4){};
				\node[below=of D4, circle,draw, inner sep=1pt, label=left:{\Large$z_3$}](D5){};
				\node[below=of D5](D50){};
				\node[right=of D50, label=left:{\Large$[\Gamma(1_1)]$}](){};
				\draw[] (D0) to (D3);
				\draw[] (D3) to (D4);
				\draw[] (D5) to (D4);
				\node[right=of Dp](X0){};
				\node[right=of X0](X01){};
				\node[right=of X01](X02){};
				\node[right=of X02](X00){};
				\node[right=of X00](X2){};\node[right=of X2](X3){};
				\node[above=of X2, circle,draw, inner sep=1pt, label=left:{\Large$(x^{-1})_3$}](W0){};
				\node[above=of W0, circle,draw, inner sep=1pt, label=left:{\Large$(x^{-1})_1$}](W1){};
				\node[above=of W1, circle,draw, inner sep=1pt, label=left:{\Large$x_1$}](W2){};
				\node[below=of W0, circle,draw, inner sep=1pt, label=left:{\Large$1_3$}](W3){};
				\node[below=of W3, circle,draw, inner sep=1pt, label=left:{\Large$1_1$}](W4){};
				\node[below=of W4, circle,draw, inner sep=1pt, label=left:{\Large$z_3$}](W5){};
				\node[below=of W5](W50){};
				\node[right=of W50, label=left:{\Large$[\Gamma(1_2)$]}](){};
				\draw[] (W0) to (W1);
				\draw[] (W3) to (W4);
				\draw[] (W5) to (W4);
				\node[right=of X2](Y00){};
				\node[right=of Y00](Y01){};
				\node[right=of Y01](Y02){};
				\node[right=of Y02](Y20){};
				\node[right=of Y20](E1){};
				\node[right=of E1](E2){};\node[right=of E2](E3){};
				\node[above=of E1, circle,draw, inner sep=1pt, label=left:{\Large$(z^{-1})_2$}](F0){};
				\node[above=of F0, circle,draw, inner sep=1pt, label=left:{\Large$(z^{-1})_1$}](F1){};
				\node[above=of F1, circle,draw, inner sep=1pt, label=left:{\Large$y_1$}](F2){};
				\node[below=of F0, circle,draw, inner sep=1pt, label=left:{\Large$1_2$}](F3){};
				\node[below=of F3, circle,draw, inner sep=1pt, label=left:{\Large$1_1$}](F4){};
				\node[below=of F4, circle,draw, inner sep=1pt, label=left:{\Large$x_2$}](F5){};
				\node[below=of F5](F50){};
				\node[right=of F50, label=left:{\Large$[\Gamma(1_3)]$}](){};
				\draw[] (F0) to (F1);
				\draw[] (F3) to (F4);
				\draw[] (F4) to (F5);
			\end{tikzpicture}
		\end{figure}
		
		Suppose that there exists $\a\in A$ with $(1_2)^{\a}=1_3$. Then $\a$ induces an isomorphism from $[\Gamma(1_2)]$ to $[\Gamma(1_3)]$.
		Hence $(1_1)^{\a}=1_1$, and then as $\a$ stabilizes $\Gamma(1_1)$, we derive that $(1_3)^{\a}=1_2$ and $(z_3)^{\a}=x_2$.
		Consequently,
		\begin{equation}\label{Eq-5}
			|\Gamma(1_3)\cap\Gamma(z_3)|=|\Gamma(1_2)\cap\Gamma(x_2)|\ \text{ and }\ |\Gamma(1_3)\cap\Gamma(z_3)\cap G_1|=|\Gamma(1_2)\cap\Gamma(x_2)\cap G_1|.
		\end{equation}
		However, since $\Gamma(z_3)=\{z_1,(yz)_1,1_1,z_2,(xz)_2,1_2\}$ and $\Gamma(x_2)=\{x_1,1_1,(x^2)_1,x_3,1_3,(zx)_3\}$, we have the following:
		\begin{itemize}
			\item if $|z|=2$ and $|x|=3$, then $\Gamma(1_3)\cap \Gamma(z_3)=\{1_1,z_1,1_2,z_2\}$ and $\Gamma(1_2)\cap \Gamma(x_2)=\{1_1,x_1,(x^2)_1,1_3\}$;
			\item if $|z|=2$ and $|x|>3$, then $\Gamma(1_3)\cap \Gamma(z_3)=\{1_1,z_1,1_2,z_2\}$ and $\Gamma(1_2)\cap \Gamma(x_2)=\{1_1,x_1,1_3\}$;
			\item if $|z|>2$ and $|x|=3$, then $\Gamma(1_3)\cap \Gamma(z_3)=\{1_1,1_2\}$ and $\Gamma(1_2)\cap \Gamma(x_2)=\{1_1,x_1,(x^2)_1,1_3\}$;
			\item if $|z|>2$ and $|x|>3$, then $\Gamma(1_3)\cap \Gamma(z_3)=\{1_1,1_2\}$ and $\Gamma(1_2)\cap \Gamma(x_2)=\{1_1,x_1,1_3\}$.
		\end{itemize}
		None of the possibilities satisfies~\eqref{Eq-5}.
		
		Thus, $A$ stabilizes $G_2$ and $G_3$ respectively. It follows that $A_{1_i}$ fixes $\Gamma(1_i)$ pointwise for $i\in\{1,2,3\}$.
		 Moreover, $A_{1_1}=A_{1_2}=A_{1_3}$. Hence $\Gamma$ is a $3$-HGR of $G$ by Lemma~\ref{Lem=stabilizertrivial}.
		
		\smallskip
		\textsf{Case 2}: $m=4$.
		\smallskip
		
		Let $\Gamma=\Cay(G,(T_{i,j})_{4\times4})$ with $T_{1,2}=T_{1,3}=\{1,x\}$, $T_{1,4}=T_{2,3}=\{1\}$, $T_{2,4}=\{1,y\}$ and $T_{3,4}=\{y,z\}$, and let $A=\Aut(\Gamma)$. Then $\Gamma$ is a $4$-Haar graph of $G$ with valency $5$ and
		\begin{align*}
			\Gamma(1_1)&=\{1_2,x_2,1_3,x_3,1_4\},\\
			\Gamma(1_2)&=\{1_1,(x^{-1})_1,1_3,1_4,y_4\},\\
			\Gamma(1_3)&=\{1_1,(x^{-1})_1,1_2,y_4,z_4\},\\
			\Gamma(1_4)&=\{1_1,1_2,(y^{-1})_2,(y^{-1})_3,(z^{-1})_3\}.
		\end{align*}
		Considering the induced subgraphs $[\Gamma(1_1)]$, $[\Gamma(1_2)]$, $[\Gamma(1_3)]$ and $[\Gamma(1_4)]$, as depicted in Figure~\ref{Fig9}, we see that $[\Gamma(1_i)]\ncong [\Gamma(1_j)]$ for distinct $i$ and $j$ in $\{1,2,3,4\}$. Hence $A$ stabilizes $G_i$ for each $i\in\{1,2,3,4\}$. Then we conclude that $A_{1_i}$ fixes $\Gamma(1_i)$ pointwise for $i\in\{1,2,4\}$, and $A_{1_3}$ fixes each of $1_2$, $y_4$ and $z_4$. In particular, $A_{1_3}\leq A_{1_2}$ fixes $1_1$. This in turn implies that $A_{1_3}$ fixes $\Gamma(1_3)$ pointwise.  Moreover, $A_{1_1}=A_{1_2}=A_{1_3}=A_{1_4}$. Therefore, $\Gamma$ is a $4$-HGR of $G$ by Lemma~\ref{Lem=stabilizertrivial}.
		
		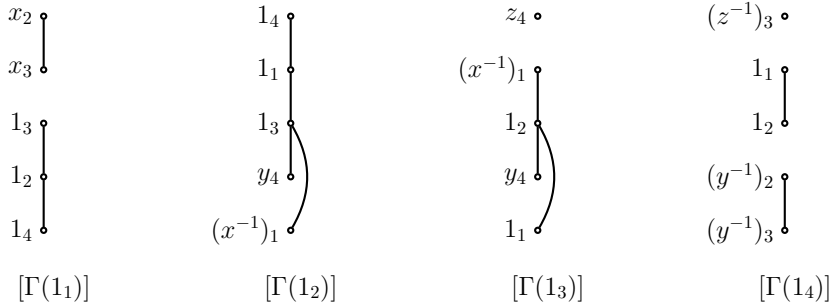
\begin{figure}[!ht]
			\vspace{3mm}
			\caption{The induced subgraph $[\Gamma(1_i)]$ with $i\in\{1,2,3,4\}$}\label{Fig9}
			\begin{tikzpicture}[node distance=0.9cm,thick,scale=0.63,every node/.style={transform shape},scale=1.1]
				\node[](A0){};
				\node[right=of A0](D00){};
				\node[right=of D00](D000){};
				\node[right=of D000](Dp){};
				\node[above=of Dp, circle,draw, inner sep=1pt, label=left:{\Large$x_3$}](D0){};
				\node[above=of D0, circle,draw, inner sep=1pt, label=left:{\Large$x_2$}](D1){};
				\node[below=of D0, circle,draw, inner sep=1pt, label=left:{\Large$1_3$}](D3){};
				\node[below=of D3, circle,draw, inner sep=1pt, label=left:{\Large$1_2$}](D4){};
				\node[below=of D4, circle,draw, inner sep=1pt, label=left:{\Large$1_4$}](D5){};
				\node[below=of D5](D50){};
				\node[right=of D50, label=left:{\Large$[\Gamma(1_1)]$}](){};
				\draw[] (D0) to (D1);
				\draw[] (D3) to (D4);
				\draw[] (D5) to (D4);
				\node[right=of Dp](X0){};
				\node[right=of X0](X01){};
				\node[right=of X01](X02){};
				\node[right=of X02](X00){};
				\node[above=of X00, circle,draw, inner sep=1pt, label=left:{\Large$1_1$}](W0){};
				\node[above=of W0, circle,draw, inner sep=1pt, label=left:{\Large$1_4$}](W1){};
				\node[below=of W0, circle,draw, inner sep=1pt, label=left:{\Large$1_3$}](W3){};
				\node[below=of W3, circle,draw, inner sep=1pt, label=left:{\Large$y_4$}](W4){};
				\node[below=of W4, circle,draw, inner sep=1pt, label=left:{\Large$(x^{-1})_1$}](W5){};
				\node[below=of W5](W50){};
				\node[right=of W50, label=left:{\Large$[\Gamma(1_2)]$}](){};
				\draw[] (W0) to (W1);
				\draw[] (W0) to (W3);
				\draw[] (W4) to (W3);
				\draw [bend left] (W3) to (W5);
				\node[right=of X00](Y0){};
				\node[right=of Y0](Y01){};
				\node[right=of Y01](Y02){};
				\node[right=of Y02](Y00){};
				\node[above=of Y00, circle,draw, inner sep=1pt, label=left:{\Large$(x^{-1})_1$}](F0){};
				\node[above=of F0, circle,draw, inner sep=1pt, label=left:{\Large$z_4$}](F1){};
				\node[below=of F0, circle,draw, inner sep=1pt, label=left:{\Large$1_2$}](F3){};
				\node[below=of F3, circle,draw, inner sep=1pt, label=left:{\Large$y_4$}](F4){};
				\node[below=of F4, circle,draw, inner sep=1pt, label=left:{\Large$1_1$}](F5){};
				\node[below=of F5](F50){};
				\node[right=of F50, label=left:{\Large$[\Gamma(1_3)]$}](){};
				\draw[] (F0) to (F3);
				\draw[] (F3) to (F4);
				\draw [bend left] (F3) to (F5);
				\node[right=of Y00](E0){};
				\node[right=of E0](E01){};
				\node[right=of E01](E02){};
				\node[right=of E02](E00){};
				\node[above=of E00, circle,draw, inner sep=1pt, label=left:{\Large$1_1$}](B0){};
				\node[above=of B0, circle,draw, inner sep=1pt, label=left:{\Large$(z^{-1})_3$}](B1){};
				\node[below=of B0, circle,draw, inner sep=1pt, label=left:{\Large$1_2$}](B3){};
				\node[below=of B3, circle,draw, inner sep=1pt, label=left:{\Large$(y^{-1})_2$}](B4){};
				\node[below=of B4, circle,draw, inner sep=1pt, label=left:{\Large$(y^{-1})_3$}](B5){};
				\node[below=of B5](B50){};
				\node[right=of B50, label=left:{\Large$[\Gamma(1_4)]$}](){};
				\draw[] (B0) to (B3);
				\draw[] (B5) to (B4);
			\end{tikzpicture}
			\vspace{3mm}
		\end{figure}
		
		\smallskip
		\textsf{Case 3}: $m\geq 5$ odd.
		\smallskip
		
		Let $\Sigma=\Cay(G,(T_{i,j})_{3\times3})$ with $T_{1,2}=\{1,x,y\}$, $T_{1,3}=\{1,y,z\}$ and $T_{2,3}=\{1,z\}$.
Then $\Sigma$ satisfies conditions~(a)--(c) of Lemma~\ref{Lem3PGSR} with
		\begin{align*}
			\Sigma(1_1)&=\{1_2,x_2,y_2,1_3,y_3,z_3\},\\
			\Sigma(1_2)&=\{1_1,(x^{-1})_1,(y^{-1})_1,1_3,z_3\},\\
			\Sigma(1_3)&=\{1_1,(y^{-1})_1,(z^{-1})_1,1_2,(z^{-1})_2\}.
		\end{align*}
		One may verify the following statements by the definition of $\Sigma$:
		\begin{itemize}
			\item $[\Sigma(1_1)]$ is a union of an isolated vertex $x_2$, an edge $\{y_2,y_3\}$ and a $2$-path $(1_3,1_2,z_3)$;
			\item $[\Sigma(1_2)]$ is a union of an isolated vertex $(x^{-1})_1$ and a $3$-path $((y^{-1})_1,1_3,1_1,z_3)$;
			\item $[\Sigma(1_3)]$ is a union of an edge $\{(z^{-1})_1,(z^{-1})_2\}$ and a $2$-path $(1_1,1_2,(y^{-1})_1)$.
		\end{itemize}
		Let $A=\Aut(\Sigma)$. We derive from the above statements that $A$ stabilizes $G_i$ for each $i\in\{1,2,3\}$, and then $A_{1_2}$ fixes both $1_1$ and $1_3$, while both $A_{1_1}$ and $A_{1_3}$ fix $1_2$. Hence $A_{1_1}=A_{1_2}=A_{1_3}$, which in turn implies that $A_{1_i}$ fixes $\Sigma(1_i)$ pointwise for $i\in\{1,2,3\}$.
		 Moreover, $A_{1_1}=A_{1_2}=A_{1_3}$.	 Consequently, $\Sigma$ is a $3$-PGSR of $G$ by Lemma~\ref{Lem=stabilizertrivial}. It then follows from Lemma~\ref{Lem3PGSR} that $G$ has an $m$-HGR.
		
		\smallskip
		\textsf{Case 4}: $m\geq 6$ even.
		\smallskip
		
		Let $\Sigma=\Cay(G,(T_{i,j})_{4\times4})$ with $T_{1,2}=\{1,x\}$, $T_{1,3}=\{1,z\}$, $T_{1,4}=\{1\}$, $T_{2,3}=\{x^{-1}\}$, $T_{2,4}=\{1,y\}$ and $T_{3,4}=\{x\}$, and let $A=\Aut(\Sigma)$. It is straightforward to verify that $\Sigma$ satisfies conditions~(a)--(c) of Lemma~\ref{Lem4PGSR} with
		\begin{align*}
			\Sigma(1_1)&=\{1_2,x_2,1_3,z_3,1_4\},\\
			\Sigma(1_2)&=\{1_1,(x^{-1})_1,(x^{-1})_3,1_4,y_4\},\\
			\Sigma(1_3)&=\{1_1,(z^{-1})_1,x_2,x_4\},\\
			\Sigma(1_4)&=\{1_1,1_2,(y^{-1})_2,(x^{-1})_3\},
		\end{align*}
		and the following statements hold:
		\begin{itemize}
			\item $[\Sigma(1_1)]$ is a union of an isolated vertex $z_3$ and two edges $\{1_2,1_4\}$ and $\{x_2,1_3\}$;
			\item $[\Sigma(1_2)]$ is a union of an isolated vertex $y_4$ and a $3$-path $((x^{-1})_1,(x^{-1})_3,1_4,1_1)$;
			\item $[\Sigma(1_3)]$ is a union of an isolated vertex $(z^{-1})_1$ and a $2$-path $(1_1,x_2,x_4)$;
			\item $[\Sigma(1_4)]$ is a union of an isolated vertex $(y^{-1})_2$ and a $2$-path $(1_1,1_2,(x^{-1})_3)$.
		\end{itemize}
		From the above statements we conclude that $A$ stabilizes $G_1$ and $G_2$ respectively.
  	 For each $g\in G$, the vertex $g_3$ has $|T_{1,3}| = 2$ neighbors in $G_1$ whereas $g_4$ has only $|T_{1,4}| = 1$ neighbor in $G_1$. This implies that $A$ stabilizes $G_3$ and $G_4$ respectively.
		 It then follows that $A_{1_i}$ fixes $\Sigma(1_i)$ pointwise for $i\in\{1,2,3,4\}$. In particular, $A_{1_1}=A_{1_2}=A_{1_3}=A_{1_4}$.
		 Hence $\Sigma$ is a $4$-PGSR of $G$ by Lemma~\ref{Lem=stabilizertrivial}, and so Lemma~\ref{Lem4PGSR} asserts that $G$ has an $m$-HGR.
	\end{proof}
	
	The following result summarizes Lemmas~\ref{lemma=d(G)=2} and~\ref{lemma=d(G)=3}.
	
	\begin{proposition}\label{prop=smallranks}
		Let $G$ be a group with $d(G)\leq3$ such that $G\notin\mathcal{G}_0$. Then $G$ has an $m$-HGR for each $m\geq 3$.
	\end{proposition}

	\section{Existence of $m$-HGRs for odd $m$}\label{Sec5}
	
	Throughout this section, let $G$ be a group with $d(G)=t\geq 4$. Note that $G$ is an elementary abelian $2$-group if and only if each minimal generating set of $G$ consists of involutions. Let $\{h_1,h_2,\ldots,h_{t}\}$ be a generating set of $G$ such that $|h_1|\geq3$ whenever $G$ is not an elementary abelian $2$-group. Denote
	\begin{align}
		S&=\{1,h_i\mid1\leq i\leq t\},\label{Eq-5.1}\\
		L&=\{1,h_1,h_2h_1^{-1},h_i\mid3\leq i\leq t\},\label{Eq-5.2}\\
		R&=\{1,h_1,h_ih_{i-1}^{-1}\mid2\leq i\leq t\},\label{Eq-5.3}\\
		T&=\{1,h_1,h_ih_{i-1}^{-1}\mid2\leq i\leq t-1\},\label{Eq-5.4}
	\end{align}
	and define graphs $\Gamma_3$ and $\Sigma_3$ as follows:
	\begin{itemize}
		\item $\Gamma_3=\Cay(G,(T_{i,j})_{3\times 3})$ with $T_{1,2}=S$, $T_{1,3}=L$ and $T_{2,3}=R$;
		\item $\Sigma_3=\Cay(G,(T_{i,j})_{3\times 3})$ with $T_{1,2}=S$, $T_{1,3}=L$ and $T_{2,3}=T$.
	\end{itemize}
	
	\begin{lemma}\label{lem=4.1}
		The graph $\Gamma_3$ is a $3$-HGR of $G$, and the graph $\Sigma_3$ is a $3$-PGSR of $G$ satisfying the conditions of Lemma~$\ref{Lem3PGSR}$.
	\end{lemma}
	
	\begin{proof}
		It is clear that $\Gamma_3$ is $(2t+2)$-regular and $\Sigma_3$ satisfies conditions~(a)--(c) of Lemma~\ref{Lem3PGSR}.
		Let $\Delta=\Gamma_3$ or $\Sigma_3$. Then
		\begin{align*}
			\Delta(1_1)&=(S\times \{2\})\cup (L\times \{3\}),\\
			\Delta(1_2)&=
			\begin{cases}
				(S^{-1}\times\{1\})\cup (R \times\{3\})&\text{ if }\Delta=\Gamma_3, \\
				(S^{-1}\times\{1\})\cup (T \times\{3\})&\text{ if }\Delta=\Sigma_3,
			\end{cases}\\
			\Delta(1_3)&=
			\begin{cases}
				(L^{-1}\times \{1\})\cup  (R^{-1}\times \{2\})&\text{ if }\Delta=\Gamma_3, \\
				(L^{-1}\times \{1\})\cup  (T^{-1}\times \{2\})&\text{ if }\Delta=\Sigma_3.
			\end{cases}
		\end{align*}
		Consider the subgraph $[\Delta(1_i)]$ induced by $\Delta(1_i)$, where $i\in\{1,2,3\}$. For the bipartite graph $[\Delta(1_1)]$ with parts $S\times \{2\}$ and $L\times \{3\}$, we have
		\begin{align*}
			\Delta(1_2)\cap (L\times \{3\}) &=\{1,h_1,h_2h_1^{-1}\}\times \{3\},\\
			\Delta((h_1,2))\cap (L\times \{3\}) &=
			\begin{cases}
				\{1,h_1\}\times \{3\}&\text{if }|h_1|=2,\\
				\{(h_1,3)\}&\text{if }|h_1|>2,
			\end{cases}\\
			\Delta((h_2,2))\cap (L\times \{3\}) &=
			\begin{cases}
				\{h_1,h_2h_1^{-1},h_3\}\times \{3\} &\text{if }h_2h_1^{-1}=h_1h_2\text{ and }|h_2h_1^{-1}|=2,\\
				\{h_2h_1^{-1},h_3\}\times \{3\} &\text{if }h_2h_1^{-1}=h_1h_2\text{ and }|h_2h_1^{-1}|>2,\\
				\{h_1,h_3\}\times \{3\} &\text{if }h_2h_1^{-1}\neq h_1h_2\text{ and }|h_2h_1^{-1}|=2,\\
				\{(h_3,3)\} &\text{if }h_2h_1^{-1}\neq h_1h_2\text{ and }|h_2h_1^{-1}|>2,
			\end{cases}\\
			\Delta((h_3,2))\cap (L\times \{3\}) &=\{h_3,h_4\}\times \{3\}\ \text{ if }t>4,\\
			\Delta((h_i,2))\cap (L\times \{3\}) &=
			\begin{cases}
				\{h_{i-1},h_i,h_{i+1}\}\times \{3\} &\text{if }|h_ih_{i-1}^{-1}|=2,\\
				\{h_i,h_{i+1}\}\times \{3\} &\text{if }|h_ih_{i-1}^{-1}|>2,
			\end{cases}
			\ \ (4\leq i\leq t-2),\\
			\Delta((h_{t-1},2))\cap (L\times \{3\})&=
			\begin{cases}
				\{h_{\max\{t-2,3\}},h_{t-1},h_t\}\times \{3\} &\text{if }|h_{t-1}h_{t-2}^{-1}|=2\text{ and }\Delta=\Gamma_3,\\
				\{h_{\max\{t-2,3\}},h_{t-1}\}\times \{3\} &\text{if }|h_{t-1}h_{t-2}^{-1}|=2\text{ and }\Delta=\Sigma_3,\\
				\{h_{t-1},h_t\}\times \{3\} &\text{if }|h_{t-1}h_{t-2}^{-1}|>2\text{ and }\Delta=\Gamma_3,\\
				\{(h_{t-1},3)\} &\text{if }|h_{t-1}h_{t-2}^{-1}|>2\text{ and }\Delta=\Sigma_3,\\
			\end{cases}\\
			\Delta((h_t,2))\cap (L\times \{3\})&=
			\begin{cases}
				\{h_t, h_{t-1}\}\times \{3\}&\text{if }|h_th_{t-1}^{-1}|=2\text{ and }\Delta=\Gamma_3,\\
				\{(h_t,3)\} &\text{if }|h_th_{t-1}^{-1}|>2\text{ or }\Delta=\Sigma_3.\\
			\end{cases}
		\end{align*}
		Similarly, for the bipartite graph $[\Delta(1_2)]$ with parts $S^{-1}\times \{1\}$ and $T_{2,3}\times \{3\}$, we have
		\begin{align*}
			\Delta(1_1)\cap (T_{2,3}\times \{3\}) &=\{1,h_1,h_2h_1^{-1}\}\times\{3\},\\
			\Delta((h_1^{-1},1))\cap (T_{2,3}\times \{3\})&=
			\begin{cases}
				\{1,h_1\}\times \{3\}&\text{if }|h_1|=2,\\
				\{1_3\}&\text{if }|h_1|>2,
			\end{cases}\\
		\end{align*}
		\begin{align*}
			\Delta((h_2^{-1},1))\cap (T_{2,3}\times \{3\}) &=
			\begin{cases}
				\{h_1,h_2h_1^{-1},h_3h_2^{-1}\}\times \{3\} &\text{if }h_2h_1^{-1}=h_1h_2\text { and }|h_2h_1^{-1}|=2,\\
				\{h_1,h_3h_2^{-1}\}\times \{3\} &\text{if }h_2h_1^{-1}=h_1h_2\text{ and }|h_2h_1^{-1}|>2,\\
				\{h_2h_1^{-1},h_3h_2^{-1}\} \times \{3\}  &\text{if }h_2h_1^{-1}\neq h_1h_2\text{ and }|h_2h_1^{-1}|=2,\\
				\{h_3h_2^{-1}\}\times \{3\} &\text{if }h_2h_1^{-1}\neq h_1h_2\text{ and }|h_2h_1^{-1}|>2,
			\end{cases}\\
			\Delta((h_3^{-1},1))\cap (T_{2,3}\times \{3\}) &=\{1,h_4h_3^{-1}\}\times \{3\}\ \text{ if }t>4,\\
			\Delta((h_i^{-1},1))\cap (T_{2,3}\times \{3\}) &=
			\begin{cases}
				\{1,h_ih_{i-1}^{-1},h_{i+1}h_i^{-1}\}\times \{3\} &\text{if }|h_ih_{i-1}^{-1}|=2,\\
				\{1,h_{i+1}h_i^{-1}\}\times \{3\} &\text{if }|h_ih_{i-1}^{-1}|>2,
			\end{cases}
			\ \ (4\leq i\leq t-2),\\
			\Delta((h_{t-1}^{-1},1))\cap (T_{2,3}\times \{3\}) &=
			\begin{cases}
				\{1,h_{t-1}h_{\max\{t-2,3\}}^{-1},h_th_{t-1}^{-1}\}\times \{3\} &\text{if }|h_{t-1}h_{t-2}^{-1}|=2\text{ and }\Delta=\Gamma_3,\\
				\{1,h_{t-1}h_{\max\{t-2,3\}}^{-1}\}\times \{3\} &\text{if }|h_{t-1}h_{t-2}^{-1}|=2\text{ and }\Delta=\Sigma_3,\\
				\{1,h_th_{t-1}^{-1}\}\times \{3\} &\text{if }|h_{t-1}h_{t-2}^{-1}|>2\text{ and }\Delta=\Gamma_3,\\
				\{1_3\} &\text{if }|h_{t-1}h_{t-2}^{-1}|>2\text{ and }\Delta=\Sigma_3,
			\end{cases}\\
			\Delta((h_t^{-1},1))\cap (T_{2,3}\times \{3\}) &=
			\begin{cases}
				\{1,h_th_{t-1}^{-1}\}\times \{3\} &\text{if } |h_{t}h_{t-1}^{-1}|=2\text{ and }\Delta=\Gamma_3,\\
				\{1_3\} &\text{if }|h_{t}h_{t-1}^{-1}|>2\text{ or }\Delta=\Sigma_3.\\
			\end{cases}
		\end{align*}

		Moreover, for the bipartite graph $[\Delta(1_3)]$ with parts $L^{-1}\times \{1\}$ and $T_{2,3}^{-1}\times \{2\}$, we have
		\begin{align*}
			\Delta(1_1)\cap (T_{2,3}^{-1}\times \{2\}) &=
			\begin{cases}
				\{1,h_1^{-1}\}\times \{2\}&\text{if }|h_1|=2,\\
				\{1_2\}&\text{if }|h_1|>2,
			\end{cases}\\
			\Delta((h_1^{-1},1))\cap (T_{2,3}^{-1}\times \{2\}) &
			=\begin{cases}
				\{1,h_1^{-1},h_1h_2^{-1}\}\times \{2\} &\text{if }|h_2h_1^{-1}|=2,\\
				\{1,h_1^{-1}\}\times \{2\} &\text{if }|h_2h_1^{-1}|>2,
			\end{cases}\\
			\Delta((h_1h_2^{-1},1))\cap (T_{2,3}^{-1}\times \{2\}) &=
			\begin{cases}
				\{h_1^{-1},h_1h_2^{-1}\}\times \{2\} &\text{if }h_2h_1^{-1}=h_1h_2,\\
				\{h_1h_2^{-1}\}\times \{2\} &\text{if }h_2h_1^{-1}\neq h_1h_2,
			\end{cases}\\
			\Delta((h_i^{-1},1))\cap (T_{2,3}^{-1}\times \{2\}) &=
			\begin{cases}
				\{1,h_{i-1}h_i^{-1},h_ih_{i+1}^{-1}\}\times \{2\} &\text{if }|h_{i+1}h_{i}^{-1}|=2,\\
				\{1,h_{i-1}h_i^{-1}\}\times \{2\} &\text{if }|h_{i+1}h_{i}^{-1}|>2,
			\end{cases}
			\ \ (3\leq i\leq t-2),\\
			\Delta((h_{t-1}^{-1},1))\cap (T_{2,3}^{-1}\times \{2\}) &=
			\begin{cases}
				\{1,h_{t-2}h_{t-1}^{-1},h_{t-1}h_{t}^{-1}\}\times \{2\} &\text{if }|h_{t}h_{t-1}^{-1}|=2\text{ and }\Delta=\Gamma_3,\\
				\{1,h_{t-2}h_{t-1}^{-1}\}\times \{2\} &\text{if }|h_{t}h_{t-1}^{-1}|>2\text{ or }\Delta=\Sigma_3,
			\end{cases}\\
			\Delta((h_t^{-1},1))\cap (T_{2,3}^{-1}\times \{2\}) &=
			\begin{cases}
				\{1,h_{t-1}h_t ^{-1}\}\times \{2\} &\text{if }\Delta=\Gamma_3,\\
				\{1_2\} &\text{if }\Delta=\Sigma_3.\\
			\end{cases}
		\end{align*}
		For $i\in \{1,2,3\}$, the structure of $[\Delta(1_i)]$ is depicted in Figure~\ref{Fig4} or Figure~\ref{Fig41} for $\Delta=\Gamma_3$ or $\Delta=\Sigma_3$ respectively (the depicted are when $t\geq5$, but those when $t=4$ are similar), where the dashed are the possible edges according to whether $|h_1|=2$, $h_2h_1^{-1}=h_1h_2$ or $|h_ih_{i-1}^{-1}|=2$. Let $A=\Aut(\Delta)$.
		It is evident that $1_3$ is the unique vertex of valency $t$ in $[\Delta(1_2)]$. This implies that $A_{1_2}\leq A_{1_3}$. Similarly, $1_2$ is the unique vertex of valency $t$ in $[\Delta(1_3)]$, and so $A_{1_3}\leq A_{1_2}$. Therefore,
		\[
		A_{1_2}=A_{1_3}.
		\]
		The lemma is proved along the three steps as follows.
		
		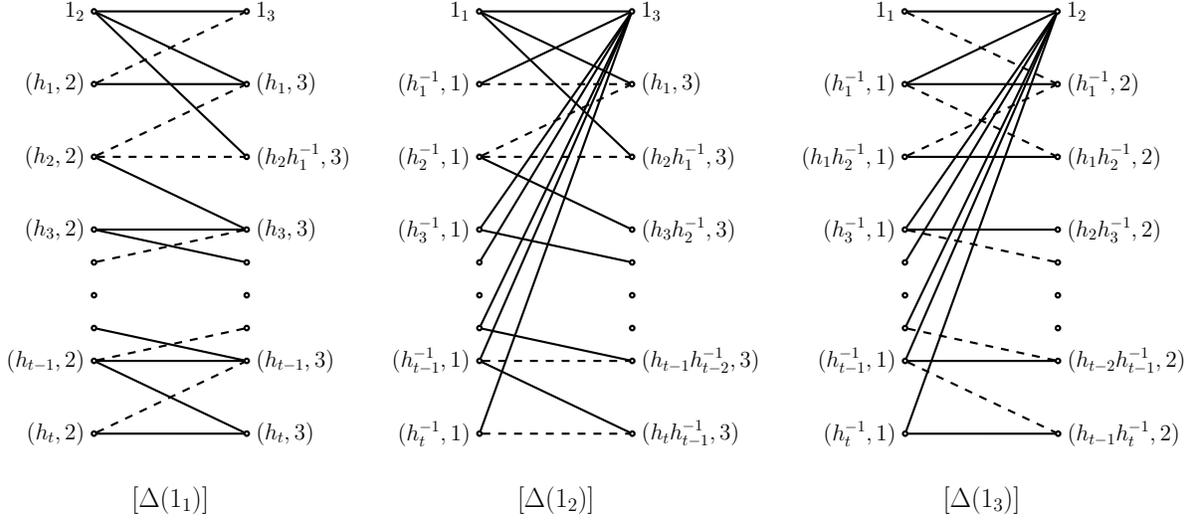
\begin{figure}[!ht]
			\vspace{3mm}
			\caption{The induced subgraphs $[\Delta(1_i)]$ with $\Delta=\Gamma_3$ and $i\in\{1,2,3\}$}\label{Fig4}
			\begin{tikzpicture}[node distance=0.6cm,thick,scale=0.6,every node/.style={transform shape},scale=1]
				\node[circle](A0){};
				\node[below=of A0, circle,draw, inner sep=1pt, label=left:{\Large$1_2$}](A1){};
				\node[below=of A1](A11){};
				\node[below=of A11, circle,draw, inner sep=1pt, label=left:{\Large$(h_1,2)$}](A2){};
				\node[below=of A2](A21){};
				\node[below=of A21, circle,draw, inner sep=1pt, label=left:{\Large$(h_2,2)$}](A3){};
				\node[below=of A3](A31){};
				\node[below=of A31, circle,draw, inner sep=1pt, label=left:{\Large$(h_3,2)$}](A4){};
				\node[below=of A4, circle,draw, inner sep=1pt, label=left:{}](A5){};
				\node[below=of A5, circle,draw, inner sep=1pt, label=left:{}](A6){};
				\node[below=of A6, circle,draw, inner sep=1pt, label=left:{}](A7){};
				\node[below=of A7, circle,draw, inner sep=1pt, label=left:{\Large$(h_{t-1},2)$}](A8){};
				\node[below=of A8](A81){};
				\node[below=of A81, circle,draw, inner sep=1pt, label=left:{\Large$(h_{t},2)$}](A9){};
				\node[right=of A1](A00){};
				\node[right=of A00](A000){};
				\node[right=of A000](A0000){};
				\node[right=of A0000, circle,draw, inner sep=1pt, label=right:{\Large$1_3$}](B1){};
				\node[below=of B1](B11){};
				\node[below=of B11, circle,draw, inner sep=1pt, label=right:{\Large$(h_1,3)$}](B2){};
				\node[below=of B2](B21){};
				\node[below=of B21, circle,draw, inner sep=1pt, label=right:{\Large$(h_2h_1^{-1},3)$}](B3){};
				\node[below=of B3](B31){};
				\node[below=of B31, circle,draw, inner sep=1pt, label=right:{\Large$(h_3,3)$}](B4){};
				\node[below=of B4, circle,draw, inner sep=1pt, label=left:{}](B5){};
				\node[below=of B5, circle,draw, inner sep=1pt, label=left:{}](B6){};
				\node[below=of B6, circle,draw, inner sep=1pt, label=left:{}](B7){};
				\node[below=of B7, circle,draw, inner sep=1pt, label=right:{\Large$(h_{t-1},3)$}](B8){};
				\node[below=of B8](B81){};
				\node[below=of B81, circle,draw, inner sep=1pt, label=right:{\Large$(h_{t},3)$}](B9){};
				\node[right=of A9](A90){};
				\node[below=of A90](A91){};
				\node[right=of A91,label=below:{\LARGE$[\Delta(1_1)]$}]{};
				\draw[] (A1) to (B1);
				\draw[] (A1) to (B2);
				\draw[] (A1) to (B3);
				\draw[] (A2) to (B2);
				\draw[dashed] (A2) to (B1);
				\draw[dashed] (A3) to (B2);
				\draw[dashed] (A3) to (B3);
				\draw[dashed] (A5) to (B4);
				\draw[] (A3) to (B4);
				\draw[] (A4) to (B4);
				\draw[] (A4) to (B5);
				\draw[] (A7) to (B8);
				\draw[] (A8) to (B8);
				\draw[] (A8) to (B9);
				\draw[dashed] (A8) to (B7);
				\draw[] (A9) to (B9);
				\draw[dashed] (A9) to (B8);
				\node[right=of B1](B11){};
				\node[right=of B11](B111){};
				\node[right=of B111](B1111){};
				\node[right=of B1111](B11111){};
				\node[right=of B11111](B111111){};
				\node[right=of B111111, circle,draw, inner sep=1pt, label=left:{\Large$1_1$}](C1){};
				\node[below=of C1](C11){};
				\node[below=of C11, circle,draw, inner sep=1pt, label=left:{\Large$(h_1^{-1},1)$}](C2){};
				\node[below=of C2](C21){};
				\node[below=of C21, circle,draw, inner sep=1pt, label=left:{\Large$(h_2^{-1},1)$}](C3){};
				\node[below=of C3](C31){};
				\node[below=of C31, circle,draw, inner sep=1pt, label=left:{\Large$(h_3^{-1},1)$}](C4){};
				\node[below=of C4, circle,draw, inner sep=1pt, label=left:{}](C5){};
				\node[below=of C5, circle,draw, inner sep=1pt, label=left:{}](C6){};
				\node[below=of C6, circle,draw, inner sep=1pt, label=left:{}](C7){};
				\node[below=of C7, circle,draw, inner sep=1pt, label=left:{\Large$(h_{t-1}^{-1},1)$}](C8){};
				\node[below=of C8](C81){};
				\node[below=of C81, circle,draw, inner sep=1pt, label=left:{\Large$(h_{t}^{-1},1)$}](C9){};
				\node[right=of C1](C00){};
				\node[right=of C00](C000){};
				\node[right=of C000](C0000){};
				\node[right=of C0000, circle,draw, inner sep=1pt, label=right:{\Large$1_3$}](D1){};
				\node[below=of D1](D11){};
				\node[below=of D11, circle,draw, inner sep=1pt, label=right:{\Large$(h_1,3)$}](D2){};
				\node[below=of D2](D21){};
				\node[below=of D21, circle,draw, inner sep=1pt, label=right:{\Large$(h_2h_1^{-1},3)$}](D3){};
				\node[below=of D3](D31){};
				\node[below=of D31, circle,draw, inner sep=1pt, label=right:{\Large$(h_3h_2^{-1},3)$}](D4){};
				\node[below=of D4, circle,draw, inner sep=1pt, label=left:{}](D5){};
				\node[below=of D5, circle,draw, inner sep=1pt, label=left:{}](D6){};
				\node[below=of D6, circle,draw, inner sep=1pt, label=left:{}](D7){};
				\node[below=of D7, circle,draw, inner sep=1pt, label=right:{\Large$(h_{t-1}h_{t-2}^{-1},3)$}](D8){};
				\node[below=of D8](D81){};
				\node[below=of D81, circle,draw, inner sep=1pt, label=right:{\Large$(h_{t}h_{t-1}^{-1},3)$}](D9){};
				\node[right=of C9](C90){};
				\node[below=of C90](C91){};
				\node[right=of C91,label=below:{\LARGE$[\Delta(1_2)]$}]{};
				\draw[] (C1) to (D1);
				\draw[] (C1) to (D2);
				\draw[] (C1) to (D3);
				\draw[] (C2) to (D1);
				\draw[dashed](C2) to (D2);
				\draw[dashed](C3) to (D2);
				\draw[dashed] (C3) to (D3);
				\draw[] (C3) to (D4);
				\draw[] (C4) to (D1);
				\draw (C4) to (D5);
				\draw (C5) to (D1);
				\draw (C7) to (D1);
				\draw (C7) to (D8);
				\draw[] (C8) to (D1);
				\draw[dashed] (C8) to (D8);
				\draw[] (C8) to (D9);
				\draw[] (C9) to (D1);
				\draw[dashed] (C9) to (D9);
				\node[right=of D1](D11){};
				\node[right=of D11](D111){};
				\node[right=of D111](D1111){};
				\node[right=of D111](D1111){};
				\node[right=of D1111](D11111){};
				\node[right=of D11111](D111111){};
				\node[right=of D111111](D1111111){};
				\node[right=of D1111111, circle,draw, inner sep=1pt, label=left:{\Large$1_1$}](E1){};
				\node[below=of E1](E11){};
				\node[below=of E11, circle,draw, inner sep=1pt, label=left:{\Large$(h_1^{-1},1)$}](E2){};
				\node[below=of E2](E21){};
				\node[below=of E21, circle,draw, inner sep=1pt, label=left:{\Large$(h_1h_2^{-1},1)$}](E3){};
				\node[below=of E3](E31){};
				\node[below=of E31, circle,draw, inner sep=1pt, label=left:{\Large$(h_3^{-1},1)$}](E4){};
				\node[below=of E4, circle,draw, inner sep=1pt, label=left:{}](E5){};
				\node[below=of E5, circle,draw, inner sep=1pt, label=left:{}](E6){};
				\node[below=of E6, circle,draw, inner sep=1pt, label=left:{}](E7){};
				\node[below=of E7, circle,draw, inner sep=1pt, label=left:{\Large$(h_{t-1}^{-1},1)$}](E8){};
				\node[below=of E8](E81){};
				\node[below=of E81, circle,draw, inner sep=1pt, label=left:{\Large$(h_{t}^{-1},1)$}](E9){};
				\node[right=of E1](E0){};
				\node[right=of E0](E00){};
				\node[right=of E00](E000){};
				\node[right=of E000](E0000){};
				\node[right=of E0000](E00000){};
				\node[right=of E000, circle,draw, inner sep=1pt, label=right:{\Large$1_2$}](F1){};
				\node[below=of F1](F11){};
				\node[below=of F11, circle,draw, inner sep=1pt, label=right:{\Large$(h_1^{-1},2)$}](F2){};
				\node[below=of F2](F21){};
				\node[below=of F21, circle,draw, inner sep=1pt, label=right:{\Large$(h_1h_2^{-1},2)$}](F3){};
				\node[below=of F3](F31){};
				\node[below=of F31, circle,draw, inner sep=1pt, label=right:{\Large$(h_2h_3^{-1},2)$}](F4){};
				\node[below=of F4, circle,draw, inner sep=1pt, label=left:{}](F5){};
				\node[below=of F5, circle,draw, inner sep=1pt, label=left:{}](F6){};
				\node[below=of F6, circle,draw, inner sep=1pt, label=left:{}](F7){};
				\node[below=of F7, circle,draw, inner sep=1pt, label=right:{\Large$(h_{t-2}h_{t-1}^{-1},2)$}](F8){};
				\node[below=of F8](F81){};
				\node[below=of F81, circle,draw, inner sep=1pt, label=right:{\Large$(h_{t-1}h_{t}^{-1},2)$}](F9){};
				\node[right=of E9](E90){};
				\node[below=of E90](E91){};
				\node[right=of E91,label=below:{\LARGE$[\Delta(1_3)]$}]{};
				\draw[] (E1) to (F1);
				\draw[] (E2) to (F1);
				\draw[] (E2) to (F2);
				\draw[dashed] (E1) to (F2);
				\draw[dashed] (E2) to (F3);
				\draw[dashed](E3) to (F2);
				\draw[] (E3) to (F3);
				\draw[] (E4) to (F1);
				\draw[] (E4) to (F4);
				\draw[dashed](E4) to (F5);
				\draw (E5) to (F1);
				\draw (E7) to (F1);
				\draw[dashed] (E7) to (F8);
				\draw[dashed](E8) to (F9);
				\draw[] (E8) to (F1);
				\draw[] (E8) to (F8);
				\draw[] (E9) to (F1);
				\draw[] (E9) to (F9);
			\end{tikzpicture}
			\vspace{3mm}
		\end{figure}
		
		\begin{figure}[!ht]
			\vspace{3mm}
			\caption{The induced subgraphs $[\Delta(1_i)]$ with $\Delta=\Sigma_3$ and $i\in\{1,2,3\}$}\label{Fig41}
			\begin{tikzpicture}[node distance=0.6cm,thick,scale=0.6,every node/.style={transform shape},scale=1]
				\node[circle](A0){};
				\node[below=of A0, circle,draw, inner sep=1pt, label=left:{\Large$1_2$}](A1){};
				\node[below=of A1](A11){};
				\node[below=of A11, circle,draw, inner sep=1pt, label=left:{\Large$(h_1,2)$}](A2){};
				\node[below=of A2](A21){};
				\node[below=of A21, circle,draw, inner sep=1pt, label=left:{\Large$(h_2,2)$}](A3){};
				\node[below=of A3](A31){};
				\node[below=of A31, circle,draw, inner sep=1pt, label=left:{\Large$(h_3,2)$}](A4){};
				\node[below=of A4, circle,draw, inner sep=1pt, label=left:{}](A5){};
				\node[below=of A5, circle,draw, inner sep=1pt, label=left:{}](A6){};
				\node[below=of A6, circle,draw, inner sep=1pt, label=left:{}](A7){};
				\node[below=of A7, circle,draw, inner sep=1pt, label=left:{\Large$(h_{t-1},2)$}](A8){};
				\node[below=of A8](A81){};
				\node[below=of A81, circle,draw, inner sep=1pt, label=left:{\Large$(h_{t},2)$}](A9){};
				\node[right=of A1](A00){};
				\node[right=of A00](A000){};
				\node[right=of A000](A0000){};
				\node[right=of A0000, circle,draw, inner sep=1pt, label=right:{\Large$1_3$}](B1){};
				\node[below=of B1](B11){};
				\node[below=of B11, circle,draw, inner sep=1pt, label=right:{\Large$(h_1,3)$}](B2){};
				\node[below=of B2](B21){};
				\node[below=of B21, circle,draw, inner sep=1pt, label=right:{\Large$(h_2h_1^{-1},3)$}](B3){};
				\node[below=of B3](B31){};
				\node[below=of B31, circle,draw, inner sep=1pt, label=right:{\Large$(h_3,3)$}](B4){};
				\node[below=of B4, circle,draw, inner sep=1pt, label=left:{}](B5){};
				\node[below=of B5, circle,draw, inner sep=1pt, label=left:{}](B6){};
				\node[below=of B6, circle,draw, inner sep=1pt, label=left:{}](B7){};
				\node[below=of B7, circle,draw, inner sep=1pt, label=right:{\Large$(h_{t-1},3)$}](B8){};
				\node[below=of B8](B81){};
				\node[below=of B81, circle,draw, inner sep=1pt, label=right:{\Large$(h_{t},3)$}](B9){};
				\node[right=of A9](A90){};
				\node[below=of A90](A91){};
				\node[right=of A91,label=below:{\LARGE$[\Delta(1_1)]$}]{};
				\draw[] (A1) to (B1);
				\draw[] (A1) to (B2);
				\draw[] (A1) to (B3);
				\draw[] (A2) to (B2);
				\draw[dashed] (A2) to (B1);
				\draw[dashed] (A3) to (B2);
				\draw[dashed] (A3) to (B3);
				\draw[] (A3) to (B4);
				\draw[] (A4) to (B4);
				\draw[dashed] (A5) to (B4);
				\draw[] (A4) to (B5);
				\draw[] (A7) to (B8);
				\draw[] (A8) to (B8);
				\draw[dashed] (A8) to (B7);
				\draw[] (A9) to (B9);
				\node[right=of B1](B11){};
				\node[right=of B11](B111){};
				\node[right=of B111](B1111){};
				\node[right=of B1111](B11111){};
				\node[right=of B11111](B111111){};
				\node[right=of B111111, circle,draw, inner sep=1pt, label=left:{\Large$1_1$}](C1){};
				\node[below=of C1](C11){};
				\node[below=of C11, circle,draw, inner sep=1pt, label=left:{\Large$(h_1^{-1},1)$}](C2){};
				\node[below=of C2](C21){};
				\node[below=of C21, circle,draw, inner sep=1pt, label=left:{\Large$(h_2^{-1},1)$}](C3){};
				\node[below=of C3](C31){};
				\node[below=of C31, circle,draw, inner sep=1pt, label=left:{\Large$(h_3^{-1},1)$}](C4){};
				\node[below=of C4, circle,draw, inner sep=1pt, label=left:{}](C5){};
				\node[below=of C5, circle,draw, inner sep=1pt, label=left:{}](C6){};
				\node[below=of C6, circle,draw, inner sep=1pt, label=left:{}](C7){};
				\node[below=of C7, circle,draw, inner sep=1pt, label=left:{\Large$(h_{t-1}^{-1},1)$}](C8){};
				\node[below=of C8](C81){};
				\node[below=of C81, circle,draw, inner sep=1pt, label=left:{\Large$(h_{t}^{-1},1)$}](C9){};
				\node[right=of C1](C00){};
				\node[right=of C00](C000){};
				\node[right=of C000](C0000){};
				\node[right=of C0000, circle,draw, inner sep=1pt, label=right:{\Large$1_3$}](D1){};
				\node[below=of D1](D11){};
				\node[below=of D11, circle,draw, inner sep=1pt, label=right:{\Large$(h_1,3)$}](D2){};
				\node[below=of D2](D21){};
				\node[below=of D21, circle,draw, inner sep=1pt, label=right:{\Large$(h_2h_1^{-1},3)$}](D3){};
				\node[below=of D3](D31){};
				\node[below=of D31, circle,draw, inner sep=1pt, label=right:{\Large$(h_3h_2^{-1},3)$}](D4){};
				\node[below=of D4, circle,draw, inner sep=1pt, label=left:{}](D5){};
				\node[below=of D5, circle,draw, inner sep=1pt, label=left:{}](D6){};
				\node[below=of D6, circle,draw, inner sep=1pt, label=left:{}](D7){};
				\node[below=of D7, circle,draw, inner sep=1pt, label=right:{\Large$(h_{t-1}h_{t-2}^{-1},3)$}](D8){};
				\node[below=of D8](D81){};
				\node[right=of C9](C90){};
				\node[below=of C90](C91){};
				\node[right=of C91,label=below:{\LARGE$[\Delta(1_2)]$}]{};
				\draw[] (C1) to (D1);
				\draw[] (C1) to (D2);
				\draw[dashed] (C2) to (D2);
				\draw[] (C1) to (D3);
				\draw[] (C2) to (D1);
				\draw[dashed](C3) to (D2);
				\draw[dashed] (C3) to (D3);
				\draw[] (C3) to (D4);
				\draw[] (C4) to (D1);
				\draw (C4) to (D5);
				\draw (C5) to (D1);
				\draw (C7) to (D1);
				\draw[] (C7) to (D8);
				\draw[] (C8) to (D1);
				\draw[dashed] (C8) to (D8);
				\draw[] (C9) to (D1);
				\node[right=of D1](D11){};
				\node[right=of D11](D111){};
				\node[right=of D111](D1111){};
				\node[right=of D111](D1111){};
				\node[right=of D1111](D11111){};
				\node[right=of D11111](D111111){};
				\node[right=of D111111](D1111111){};
				\node[right=of D1111111, circle,draw, inner sep=1pt, label=left:{\Large$1_1$}](E1){};
				\node[below=of E1](E11){};
				\node[below=of E11, circle,draw, inner sep=1pt, label=left:{\Large$(h_1^{-1},1)$}](E2){};
				\node[below=of E2](E21){};
				\node[below=of E21, circle,draw, inner sep=1pt, label=left:{\Large$(h_1h_2^{-1},1)$}](E3){};
				\node[below=of E3](E31){};
				\node[below=of E31, circle,draw, inner sep=1pt, label=left:{\Large$(h_3^{-1},1)$}](E4){};
				\node[below=of E4, circle,draw, inner sep=1pt, label=left:{}](E5){};
				\node[below=of E5, circle,draw, inner sep=1pt, label=left:{}](E6){};
				\node[below=of E6, circle,draw, inner sep=1pt, label=left:{}](E7){};
				\node[below=of E7, circle,draw, inner sep=1pt, label=left:{\Large$(h_{t-1}^{-1},1)$}](E8){};
				\node[below=of E8](E81){};
				\node[below=of E81, circle,draw, inner sep=1pt, label=left:{\Large$(h_{t}^{-1},1)$}](E9){};
				\node[right=of E1](E0){};
				\node[right=of E0](E00){};
				\node[right=of E00](E000){};
				\node[right=of E000](E0000){};
				\node[right=of E0000](E00000){};
				\node[right=of E000, circle,draw, inner sep=1pt, label=right:{\Large$1_2$}](F1){};
				\node[below=of F1](F11){};
				\node[below=of F11, circle,draw, inner sep=1pt, label=right:{\Large$(h_1^{-1},2)$}](F2){};
				\node[below=of F2](F21){};
				\node[below=of F21, circle,draw, inner sep=1pt, label=right:{\Large$(h_1h_2^{-1},2)$}](F3){};
				\node[below=of F3](F31){};
				\node[below=of F31, circle,draw, inner sep=1pt, label=right:{\Large$(h_2h_3^{-1},2)$}](F4){};
				\node[below=of F4, circle,draw, inner sep=1pt, label=left:{}](F5){};
				\node[below=of F5, circle,draw, inner sep=1pt, label=left:{}](F6){};
				\node[below=of F6, circle,draw, inner sep=1pt, label=left:{}](F7){};
				\node[below=of F7, circle,draw, inner sep=1pt, label=right:{\Large$(h_{t-2}h_{t-1}^{-1},2)$}](F8){};
				\node[below=of F8](F81){};
				\node[right=of E9](E90){};
				\node[below=of E90](E91){};
				\node[right=of E91,label=below:{\LARGE$[\Delta(1_3)]$}]{};
				\draw[] (E1) to (F1);
				\draw[dashed] (E1) to (F2);
				\draw[] (E2) to (F1);
				\draw[] (E2) to (F2);
				\draw[dashed] (E2) to (F3);
				\draw[dashed](E3) to (F2);
				\draw[] (E3) to (F3);
				\draw[] (E4) to (F1);
				\draw[] (E4) to (F4);
				\draw (E5) to (F1);
				\draw (E7) to (F1);
				\draw[dashed](E4) to (F5);
				\draw[dashed] (E7) to (F8);
				\draw[] (E8) to (F1);
				\draw[] (E8) to (F8);
				\draw[] (E9) to (F1);
			\end{tikzpicture}
			\vspace{3mm}
		\end{figure}
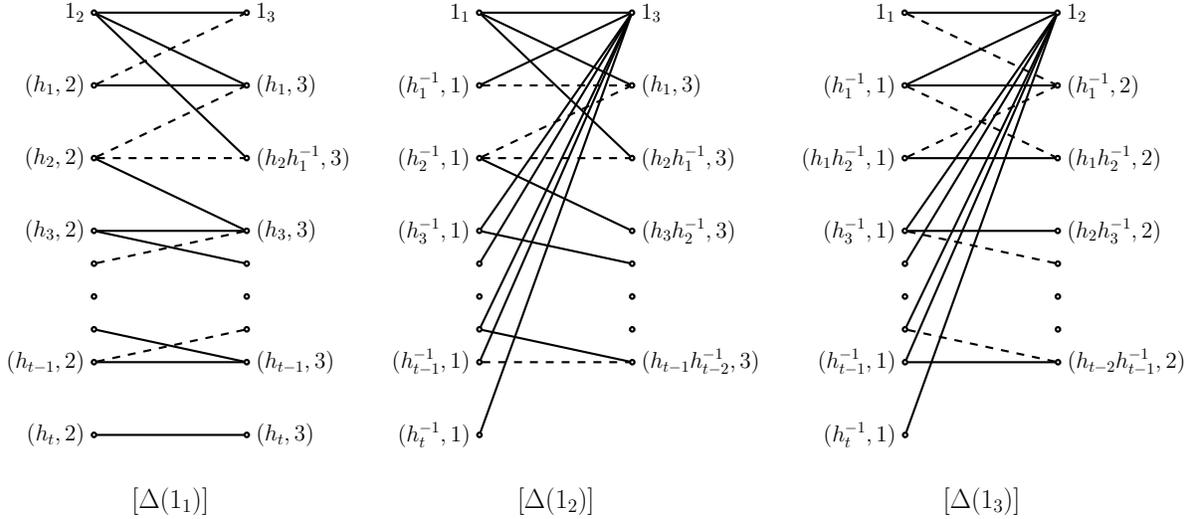
		
		\smallskip
		\textsf{Step 1}: prove that $A$ stabilizes $G_i$ for each $i\in\{1,2,3\}$.
		\smallskip
		
		Since the vertices in $[\Delta(1_1)]$ all have valency at most $3$, we have $[\Delta(1_1)]\ncong[\Delta(1_i)]$ for $i\in\{2,3\}$. Hence $A$ stabilizes $G_1$. To complete Step~1, suppose for a contradiction that exists $\alpha\in A$ with $(1_2)^{\alpha}=1_3$.
		
		As $1_{5-i}$ is the unique vertex of valency $t$ in $[\Delta(1_i)]$ for $i\in\{2,3\}$, it follows that $(1_3)^{\alpha}=1_2$.
		Hence $\alpha$ swaps $\Delta(1_2)$ and $\Delta(1_3)$.
		Observe that $(h_2^{-1},1)$ is the unique vertex in $G_1$ not adjacent to $1_3$ in $[\Delta(1_2)]$, and that
		$(h_1h_2^{-1},1)$ is the unique vertex in $G_1$ not adjacent to $1_2$ in $[\Delta(1_3)]$. We then obtain $(h_2^{-1},1)^{\alpha}=(h_1h_2^{-1},1)$. If $|h_1h_2^{-1}|=2$, then no matter $h_2h_1^{-1}$ equals $h_1h_2$ or not, the valency of $(h_2^{-1},1)$ in $[\Delta(1_2)]$ is always one larger than the valency of $(h_1h_2^{-1},1)$ in $[\Delta(1_3)]$, not possible. Thus,
		\[
		|h_1h_2^{-1}|>2.
		\]
		Note that this also implies $|h_1|>2$, as $|h_1|=2$ if and only if $G$ is an elementary abelian $2$-group. If $h_2h_1^{-1}=h_1h_2$, then the two neighbors of $(h_2^{-1},1)$ in $[\Delta(1_2)]$ are $(h_1,3)$ and $(h_3h_2^{-1},3)$, and the two neighbors of $(h_1h_2^{-1},1)$ in $[\Delta(1_3)]$ are $(h_1^{-1},2)$ and $(h_1h_2^{-1},2)$. However, in this case, $(h_1,3)$ has a neighbor of valency $3$ in $[\Delta(1_2)]$, while neither $(h_1^{-1},2)$ nor $(h_1h_2^{-1},2)$ has a neighbor of valency $3$ in $[\Delta(1_3)]$, again impossible. Therefore, $h_2h_1^{-1}\neq h_1h_2$.
		
		First assume $\Delta=\Gamma_3$. If $|h_th_{t-1}^{-1}|>2$, then $[\Delta(1_2)]$ has exactly three vertices of valency $1$ in $G_1$, while $[\Delta(1_3)]$ has exactly two vertices of valency $1$ in $G_1$, not possible. Thereby we have $|h_th_{t-1}^{-1}|=2$. Now $(h_1^{-1},1)$ and $(h_2^{-1},1)$ are the two vertices in $G_1$ that has valency $1$ in $[\Delta(1_2)]$, while $1_1$ and $(h_1h_2^{-1},1)$ are the two vertices in $G_1$ that has valency $1$ in $[\Delta(1_3)]$. Hence $\alpha$ maps $(h_1^{-1},1)$ to $1_1$ as $(h_2^{-1},1)^{\a}=(h_1h_2^{-1},1)$. However, $(h_1^{-1},1)$ has valency $2$ in $[\Delta(1_3)]$, but $1_1$ has valency $3$ in $[\Delta(1_2)]$. This contradicts the conclusion that $\alpha$ maps $\Delta(1_3)$ to $\Delta(1_2)$.
		
		Next let $\Delta=\Sigma_3$. If $|h_{t-1}h_{t-2}^{-1}|>2$, then $[\Delta(1_2)]$ has exactly four vertices of valency $1$ in $G_1$, but $[\Delta(1_3)]$ has exactly three vertices of valency $1$ in $G_1$, not possible. Thus, $|h_{t-1}h_{t-2}^{-1}|=2$.
		It follows that $(h_1^{-1},1)$ and $(h_t^{-1},1)$ are the only vertices in $G_1$ that have valency $1$ in $[\Delta(1_2)]$, while $1_1$ and $(h_t^{-1},1)$ are the only vertices in $G_1$ that have valency $1$ in $[\Delta(1_3)]$. Hence $\alpha$ maps $\{(h_1^{-1},1),(h_t^{-1},1)\}$ to $\{1_1,(h_t^{-1},1)\}$. However, $(h_1^{-1},1)$ and $(h_t^{-1},1)$ both have valency smaller than $3$ in $[\Delta(1_3)]$, while $1_1$ has valency $3$ in $[\Delta(1_2)]$. This contradicts the conclusion that $\alpha$ maps $\Delta(1_3)$ to $\Delta(1_2)$. We then conclude that there exists no $\alpha\in A$ with $(1_2)^{\alpha}=1_3$, and so $A$ stabilizes $G_i$ for each $i\in\{1,2,3\}$.
		
		\smallskip
		\textsf{Step 2}: prove $A_{1_1}\leq A_{1_3}$.
		\smallskip
		
		First assume that $|h_1|=2$, that is, $G$ is an elementary abelian $2$-group. If $\Delta=\Gamma_3$, then $1_3$, $(h_2h_1^{-1},3)$ and $(h_t,3)$ are the only three vertices in $G_3$ of valency $2$ in $[\Delta(1_1)]$. If $\Delta=\Sigma_3$, then $1_3$, $(h_2h_1^{-1},3)$ and $(h_{t-1},3)$ are the only three vertices in $G_3$ of valency $2$ in $[\Delta(1_1)]$.
		In either case, among the three vertices, only $(h_2h_1^{-1},3)$ has two neighbors in $G_2$ of valency $3$ in $[\Delta(1_1)]$, and only $1_3$ and $(h_2h_1^{-1},3)$ have a common neighbor in $[\Delta(1_1)]$. Therefore, $A_{1_1}$ fixes $1_3$, as required.
		
		Next assume that $|h_1|>2$. If $h_2h_1^{-1}=h_1h_2$, then $1_3$ is the unique vertex of valency $1$ in $[\Delta(1_1)]$ that has a neighbor of valency $3$ in $[\Delta(1_1)]$, whence $A_{1_1}$ fixes $1_3$, as required. Now assume that $h_2h_1^{-1}\neq h_1h_2$.
		Then $1_3$ and $(h_2h_1^{-1},3)$ are the only two vertices of valency $1$ in $[\Delta(1_1)]$ that has a neighbor of valency $3$ in $[\Delta(1_1)]$. Consequently, $A_{1_1}$ stabilizes $\{1_3,(h_2h_1^{-1},3)\}$.
		Since $1_2$ is the unique common neighbor to both of $1_3$ and $(h_2h_1^{-1},3)$ in $[\Delta(1_1)]$, it follows that $A_{1_1}$ fixes $1_2$, and so $A_{1_1}\leq A_{1_2}=A_{1_3}$.
		
		\smallskip
		\textsf{Step 3}: prove $A_{1_1}=1$.
		\smallskip
		
		Since $A_{1_2}=A_{1_3}$ and since $(h_2h_1^{-1},1)$ is the unique vertex in $[\Delta(1_3)]\cap G_1$ that is not adjacent to $1_2$ in $[\Delta(1_3)]$, we derive that $A_{1_3}$ fixes $(h_2h_1^{-1},1)$, that is, $A_{1_3}\leq A_{(h_2h_1^{-1},1)}$. Since $A\geq G$ is transitive on $G_1$, we have $|A_{(h_2h_1^{-1},1)}|=|A_{1_1}|$.
		
		Recall from Step 1 that $A$ stabilizes $G_1$, $G_2$ and $G_3$ respectively. Then for each $i\in\{1,2,3\}$, it follows from $G \leq  A$ that $A$ acts transitively on $G_i$, and so $|A_{1_i}|=|A|/|G|$.
		 This in conjunction with $A_{1_1}\leq A_{1_2}=A_{1_3}$ leads to $A_{1_1}=A_{1_2}=A_{1_3}$. Then by the transitivity of $A$ on $G_i$ for each $i\in\{1,2,3\}$, we conclude that $A_{(g,1)}=A_{(g,2)}=A_{(g,3)}$ for each $g\in G$.
		
		From $A_{1_1}=A_{1_3}\leq A_{(h_2h_1^{-1},1)}=A_{(h_2h_1^{-1},3)}$ we see that $A_{1_1}$ fixes $(h_2h_1^{-1},3)$. Then since $1_3$, $(h_1,3)$ and $(h_2h_1^{-1},3)$ are the only neighbors of $1_2$ in $[\Delta(1_1)]$, it follows that $A_{1_1}$ fixes $(h_1,3)$.
		Moreover, we deduce from $A_{1_3}\leq A_{(h_2h_1^{-1},1)}$ that
		\[
		A_{(h_1,3)}=A_{(1\cdot h_1,3)}\leq A_{(h_2h_1^{-1}\cdot h_1,1)}=A_{(h_2,1)}=A_{(h_2,2)}.
		\]
		Hence $A_{1_1}\leq A_{(h_1,3)}\leq A_{(h_2,2)}$. Now $A_{1_1}$ fixes $(h_2,2)$, while $(h_3,3)$ is the only neighbor of $(h_2,2)$ in $\Delta(1_1)$ other than $(h_1,3)$ and $(h_2h_1^{-1},3)$. It follows that $A_{1_1}\leq A_{(h_3,3)}=A_{(h_3,2)}$. This in turn implies that $A_{1_1}$ fixes $(h_4,3)$, the only neighbor of $(h_3,2)$ in $\Delta(1_1)$ other than $(h_3,3)$. Repeating this argument, we obtain that $A_{1_1}$ fixes $(h_i,3)$ for every $i\in\{4,\ldots,t\}$.
		
		Thus far we have shown that $A_{1_1}$ fixes $(x,3)$ for all $x\in L$. Consequently, $A_{(y,1)}\leq A_{(xy,3)}=A_{(xy,1)}$ for all $x\in L$ and $y\in G$. Since $G=\langle L\rangle$, it follows that $A_{1_1}\leq A_{(g,1)}$ for all $g\in G$. Then as $A_{(g,1)}=A_{(g,2)}=A_{(g,3)}$, we conclude that $A_{1_1}$ fixes every vertex of $\Delta$. Hence $A_{1_1}=1$, completing the proof.
	\end{proof}
	
	We are now in a position to prove Theorem~\ref{theo=main} in the case when $m$ is odd.
	
	\begin{proof}[Proof of Theorem~$\ref{theo=main}$ for odd $m$]
		From Propositions~\ref{prop=smallgroups} and~\ref{prop=smallranks} we see that Theorem~\ref{theo=main} holds for $d(G)\leq3$. Now assume that $d(G)\geq4$. Then as Lemma~\ref{lem=4.1} asserts, $G$ has a $3$-HGR $\Gamma_3$ and a $3$-PGSR $\Sigma_3$ satisfying the conditions of Lemma~\ref{Lem3PGSR}. The latter further implies that $G$ has an $m$-HGR for each odd $m\geq5$. Hence Theorem~\ref{theo=main} holds for every odd $m\geq3$.
	\end{proof}

	\section{Existence of $m$-HGRs for even $m$}\label{Sec6}
	
	With the same notation as in the proceeding section, let $G$ be a group with $d(G)=t\geq 4$, and let $\{h_1,h_2,\ldots,h_{t}\}$ be a generating set of $G$ such that $|h_1|\geq3$ whenever $G$ is not an elementary abelian $2$-group.
	Take subsets $S$, $L$, $R$ of $G$ as in~\eqref{Eq-5.1}--\eqref{Eq-5.4}, and take
	\begin{align*}
		M&=\{1,h_i\mid1\leq i\leq t-1\},\\
		N&=\{h_1h_2h_3h_4,h_1h_3h_4,h_2h_3h_4,h_ih_2h_1^{-1}\mid3\leq i\leq t\}.
	\end{align*}
	Note that $|M|=t$ and $|N|=t+1$. Define graphs $\Gamma_4$ and $\Sigma_4$ as follows:
	\begin{itemize}
		\item $\Gamma_4=\Cay(G,(T_{i,j})_{4\times 4})$ with $T_{1,2}=T_{1,4}=T_{3,4}=S$, $T_{1,3}=L$, $T_{2,3}=R$ and $T_{2,4}=N$;
		\item  $\Sigma_4=\Cay(G,(T_{i,j})_{4\times 4})$ with $T_{1,2}=T_{1,4}=S$, $T_{1,3}=L$, $T_{2,3}=R$, $T_{2,4}=N$ and $T_{3,4}=M$.
	\end{itemize}
	
	\begin{lemma}\label{Lem6.3}
		The graph $\Gamma_4$ is a $4$-HGR of $G$, and $\Sigma_4$ is a $4$-PGSR of $G$ satisfying the conditions of Lemma~$\ref{Lem4PGSR}$.
	\end{lemma}
	
	\begin{proof}
		It is straightforward to verify that $\Gamma_4$ is $(3t+3)$-regular and that conditions~(a)--(c) of Lemma~\ref{Lem4PGSR} are satisfied for $\Sigma_4$. If $G=C_2^4$, then computation in \textsc{Magma}~\cite{magma} confirms that $\Gamma_4$ is a $4$-HGR and $\Sigma_4$
		is a $4$-PGSR. For the rest of the proof, assume that $G\neq C_2^4$. Let $\Delta=\Gamma_4$ or $\Sigma_4$, and let $A=\Aut(\Delta)$. Then
		\begin{align*}
			\Delta(1_1)&=(S\times \{2\})\cup (L\times \{3\})\cup (S\times \{4\}),\\
			\Delta(1_2)&=(S^{-1}\times\{1\})\cup (R \times\{3\})\cup (N\times \{4\}),\\
			\Delta(1_3)&=\begin{cases}
				(L^{-1}\times \{1\})\cup  (R^{-1}\times \{2\}) \cup (S\times \{4\})&\text{if }\Delta=\Gamma_4,\\
				(L^{-1}\times \{1\})\cup  (R^{-1}\times \{2\})\cup (M\times \{4\})&\text{if }\Delta=\Sigma_4,\\
			\end{cases}\\
			\Delta(1_4)&=\begin{cases}
				(S^{-1}\times \{1\})\cup  (N^{-1}\times \{2\}) \cup (S^{-1}\times \{3\})&\text{if }\Delta=\Gamma_4,\\
				(S^{-1}\times \{1\})\cup  (N^{-1}\times \{2\})\cup (M^{-1}\times \{3\})&\text{if }\Delta=\Sigma_4.
			\end{cases}
		\end{align*}
		The key to the proof is analyzing the induced subgraph $[\Delta(1_i)]$ for $i\in\{1,2,3,4\}$.
		
		First consider $[\Delta(1_1)]$. Note that the induced subgraph $[G_1\cup G_2\cup G_3]$ is the graph $\Gamma_3$ defined in Section~\ref{Sec5}. In particular, the subgraph induced by $(S\times \{2\})\cup (L\times \{3\})$ in $[\Delta(1_1)]$ is the same as
		$[\Gamma_3(1_1)]$ (see the first graph in Figure~\ref{Fig4} for $t\geq 5$). Since $T_{2,4}=N$, there is no edge between $S\times \{2\}$ and $S\times \{4\}$ in $[\Delta(1_1)]$. For the edges between $L\times\{3\}$ and $S\times \{4\}$ in $[\Delta(1_1)]$, since $T_{3,4}=S$ when $\Delta=\Gamma_4$ and $T_{3,4}=M$ when $\Delta=\Sigma_4$, we have
		\begin{align*}
			\Delta(1_3)\cap (S\times \{4\}) &=
			\begin{cases}
				\{1,h_i\mid1\leq i\leq t \}\times \{4\}  &\text{if } \Delta=\Gamma_4,\\
				\{1,h_i\mid1\leq i\leq t-1 \}\times \{4\}  &\text{if } \Delta=\Sigma_4,\\
			\end{cases}\\
			\Delta((h_i,3))\cap (S\times \{4\}) &=
			\begin{cases}
				\{1,h_i\}\times \{4\} &\text{if }|h_i|=2,\\
				\{(h_i,4)\} &\text{if }|h_i|>2,
			\end{cases}
			\ \ (i=1\text{ or }3\leq i\leq t-1),\\
			\Delta((h_2h_1^{-1},3))\cap (S\times \{4\}) &\subseteq
			\{h_1,h_2\}\times \{4\},\\
			\Delta((h_t,3))\cap (S\times \{4\}) &=
			\begin{cases}
				\{1,h_t\}\times \{4\} &\text{if }|h_t|=2\text{ and }\Delta=\Gamma_4 ,\\
				\{(h_t,4)\} &\text{if }|h_t|>2\text{ or }\Delta=\Sigma_4.\\
			\end{cases}
		\end{align*}
		
		Consider $[\Delta(1_2)]$ in the same vein as above. The subgraph induced by $(S^{-1}\times \{1\})\cup (R\times \{3\})$ in $[\Delta(1_2)]$ is the same as $[\Gamma_3(1_2)]$  (see the second graph in Figure~\ref{Fig4} for $t\geq 5$).
		Since $T_{1,4}=S$, there is no edge between $S^{-1}\times \{1\}$ and $N\times \{4\}$ in $[\Delta(1_2)]$. For the edges between $R\times\{3\}$ and $N\times \{4\}$ in $[\Delta(1_2)]$, we have
		\begin{align*}
			\Delta(1_3)\cap (N\times\{4\})&=\Delta((h_1,3))\cap(N\times\{4\})=\emptyset,\\
			\Delta((h_2h_1^{-1},3))\cap (N\times \{4\}) &=
			\begin{cases}
				\{h_ih_2h_1^{-1}\mid3\leq i\leq t\}\times \{4\}&\text{if }\Delta=\Gamma_4,\\
				\{h_ih_2h_1^{-1}\mid3\leq i\leq t-1\}\times \{4\}&\text{if }\Delta=\Sigma_4,\\
			\end{cases}\\
			\Delta((h_3h_2^{-1},3))\cap (N\times \{4\}) &\subseteq
			\{h_3h_2h_1^{-1},h_2h_3h_4\}\times \{4\},\\
			\Delta((h_4h_3^{-1},3))\cap (N\times \{4\}) &\subseteq
			\{h_1h_3h_4,h_2h_3h_4\}\times \{4\},\\
			\Delta((h_ih_{i-1}^{-1},3))\cap(N\times\{4\})&=\emptyset,\ \ (5\leq i\leq t).
		\end{align*}
		
		Next, let us consider $[\Delta(1_4)]$. Since $T_{1,2}=S$, the graph $[\Delta(1_4)]$ has no edge between $S^{-1}\times \{1\}$ and $N^{-1}\times \{2\}$. For its edges between $T_{3,4}^{-1}\times\{3\}$ and $S^{-1}\times \{1\}$ and edges between $T_{3,4}^{-1}\times\{3\}$ and $N^{-1}\times \{2\}$, we have
		\begin{align*}
			\Delta(1_3)\cap (S^{-1}\times \{1\}) &=\{1,h_1^{-1}, h_i^{-1}\mid3\leq i\leq t\},\\
			\Delta((h_1^{-1},3))\cap (S^{-1}\times \{1\}) &\subseteq
			\{1,h_1^{-1},h_2^{-1}\}\times \{1\},\\
			\Delta((h_2^{-1},3))\cap (S^{-1}\times \{1\}) &\subseteq
			\{h_1^{-1},h_2^{-1}\}\times \{1\},\\
			\Delta((h_i^{-1},3))\cap (S^{-1}\times \{1\}) &\subseteq
			\{1,h_i^{-1}\}\times \{1\}\ \ (3\leq i\leq t-1),\\
			\Gamma_4((h_t^{-1},3))\cap (S^{-1}\times \{1\}) &=
			\begin{cases}
				\{1,h_t^{-1}\}\times \{1\}&\text{if } |h_t|=2,\\
				\{(h_t^{-1},1)\}&\text{if } |h_t|>2,
			\end{cases}\\
			\Delta(1_3)\cap (N^{-1}\times \{2\}) &=\emptyset,\\
			\Delta((h_1^{-1},3))\cap (N^{-1}\times \{2\}) &\subseteq
			\{h_4^{-1}h_3^{-1}h_1^{-1},h_1h_2^{-1}h_3^{-1}\}\times \{2\},\\
			\Delta((h_2^{-1},3))\cap (N^{-1}\times \{2\}) &\subseteq
			\{h_4^{-1}h_3^{-1}h_2^{-1}\}\times \{2\},\\
			\Delta((h_i^{-1},3))\cap (N^{-1}\times \{2\}) &=
			\{h_1h_2^{-1}h_i^{-1}\}\times \{2\},\ \ (i=3\text{ or }5\leq i\leq t-1),\\
			\Delta((h_4^{-1},3))\cap (N^{-1}\times \{2\}) &\subseteq
			\{h_4^{-1}h_3^{-1}h_2^{-1}, h_1h_2^{-1}h_4^{-1}\}\times \{2\}\ \text{ if } t\geq 5 \text{ or } \Delta=\Gamma_4,\\
			\Gamma_4((h_t^{-1},3))\cap (N^{-1}\times \{2\}) &=
			\{h_1h_2^{-1}h_t^{-1}\}\times \{2\}.
		\end{align*}
		
		For $i\in \{1,2,4\}$, the structure of $[\Gamma_4(1_i)]$ is depicted in Figure~\ref{Fig61} or Figure~\ref{Fig62} (the depicted are for $t\geq5$, but those for $t=4$ are similar), where the dashed are possible edges.
		Then one can obtain $[\Sigma_4(1_1)]$ from $[\Gamma_4(1_1)]$ by deleting the edges $\{1_3,(h_t,4)\}$ and $\{(h_t,3),1_4\}$.
		Similarly, $[\Sigma_4(1_2)]$ can be obtained from $[\Gamma_4(1_2)]$ by deleting the edge $\{(h_2h_1^{-1},3),(h_th_2h_1^{-1},4)\}$,
		and $[\Sigma_4(1_4)]$ can be obtained from $[\Gamma_4(1_4)]$ by deleting the vertex $\{(h_t^{-1},3)\}$ together with incident edges.
		We prove the lemma along the two steps as follows.
		
		\begin{figure}[!ht]
			\vspace{3mm}
			\caption{The induced subgraphs $[\Gamma_4(1_1)]$ and $[\Gamma_4(1_2)]$ }\label{Fig61}
			\begin{tikzpicture}[node distance=0.6cm,thick,scale=0.6,every node/.style={transform shape},scale=1]
				\node[circle](A0){};
				\node[right=of A0](AA){};
				\node[below=of AA, circle,draw, inner sep=1pt, label=left:{\Large$1_2$}](A1){};
				\node[below=of A1](A11){};
				\node[below=of A11, circle,draw, inner sep=1pt, label=left:{\Large$(h_1,2)$}](A2){};
				\node[below=of A2](A21){};
				\node[below=of A21, circle,draw, inner sep=1pt, label=left:{\Large$(h_2,2)$}](A3){};
				\node[below=of A3](A31){};
				\node[below=of A31, circle,draw, inner sep=1pt, label=left:{\Large$(h_3,2)$}](A4){};
				\node[below=of A4, circle,draw, inner sep=1pt, label=left:{}](A5){};
				\node[below=of A5, circle,draw, inner sep=1pt, label=left:{}](A6){};
				\node[below=of A6, circle,draw, inner sep=1pt, label=left:{}](A7){};
				\node[below=of A7, circle,draw, inner sep=1pt, label=left:{\Large$(h_{t-1},2)$}](A8){};
				\node[below=of A8](A81){};
				\node[below=of A81, circle,draw, inner sep=1pt, label=left:{\Large$(h_{t},2)$}](A9){};
				\node[right=of A1](A00){};
				\node[right=of A00](A000){};
				\node[right=of A000](A0000){};
				\node[right=of A0000, circle,draw, inner sep=1pt, label=above:{\Large$1_3$}](B1){};
				\node[below=of B1](B11){};
				\node[below=of B11, circle,draw, inner sep=1pt, label=below:{\Large$(h_1,3)$}](B2){};
				\node[below=of B2](B21){};
				\node[below=of B21, circle,draw, inner sep=1pt, label=below:{\Large$(h_2h_1^{-1},3)$}](B3){};
				\node[below=of B3](B31){};
				\node[below=of B31, circle,draw, inner sep=1pt, label=below:{\Large$(h_3,3)$}](B4){};
				\node[below=of B4, circle,draw, inner sep=1pt, label=below:{}](B5){};
				\node[below=of B5, circle,draw, inner sep=1pt, label=below:{}](B6){};
				\node[below=of B6, circle,draw, inner sep=1pt, label=below:{}](B7){};
				\node[below=of B7, circle,draw, inner sep=1pt, label=below:{\Large$(h_{t-1},3)$}](B8){};
				\node[below=of B8](B81){};
				\node[below=of B81, circle,draw, inner sep=1pt, label=below:{\Large$(h_{t},3)$}](B9){};
				\node[below=of B9](B91){};
				\node[below=of B91,label=below:{\LARGE$[\Gamma_4(1_1)]$}]{};
				\draw[] (A1) to (B1);
				\draw[] (A1) to (B2);
				\draw[] (A1) to (B3);
				\draw[] (A2) to (B2);
				\draw[dashed] (A2) to (B1);
				\draw[dashed] (A3) to (B2);
				\draw[dashed] (A3) to (B3);
				\draw[dashed] (A5) to (B4);
				\draw[] (A3) to (B4);
				\draw[] (A4) to (B4);
				\draw[] (A4) to (B5);
				\draw[] (A7) to (B8);
				\draw[] (A8) to (B8);
				\draw[] (A8) to (B9);
				\draw[dashed] (A8) to (B7);
				\draw[] (A9) to (B9);
				\draw[dashed] (A9) to (B8);
				\node[right=of B1](B11){};
				\node[right=of B11](B111){};
				\node[right=of B111](B1111){};
				\node[right=of B1111](B11111){};
				\node[right=of B11111](B111111){};
				\node[right=of B1111, circle,draw, inner sep=1pt, label=right:{\Large$1_4$}](C1){};
				\node[below=of C1](C11){};
				\node[below=of C11, circle,draw, inner sep=1pt, label=right:{\Large$(h_1,4)$}](C2){};
				\node[below=of C2](C21){};
				\node[below=of C21, circle,draw, inner sep=1pt, label=right:{\Large$(h_2,4)$}](C3){};
				\node[below=of C3](C31){};
				\node[below=of C31, circle,draw, inner sep=1pt, label=right:{\Large$(h_3,4)$}](C4){};
				\node[below=of C4, circle,draw, inner sep=1pt, label=right:{}](C5){};
				\node[below=of C5, circle,draw, inner sep=1pt, label=right:{}](C6){};
				\node[below=of C6, circle,draw, inner sep=1pt, label=right:{}](C7){};
				\node[below=of C7, circle,draw, inner sep=1pt, label=right:{\Large$(h_{t-1},4)$}](C8){};
				\node[below=of C8](C81){};
				\node[below=of C81, circle,draw, inner sep=1pt, label=right:{\Large$(h_{t},4)$}](C9){};
				\draw[] (B1) to (C1);
				\draw[] (B1) to (C2);
				\draw[] (B1) to (C3);
				\draw[] (B1) to (C4);
				\draw[] (B1) to (C5);
				\draw[] (B1) to (C6);
				\draw[] (B1) to (C8);
				\draw[] (B1) to (C9);
				\draw[] (B2) to (C2);
				\draw[dashed] (B2) to (C1);
				\draw[dashed] (B4) to (C1);
				\draw[dashed] (B5) to (C1);
				\draw[dashed] (B7) to (C1);
				\draw[dashed] (B8) to (C1);
				\draw[dashed] (B9) to (C1);
				\draw[dashed] (B3) to (C2);
				\draw[dashed] (B3) to (C3);
				\draw[] (B4) to (C4);
				\draw[] (B8) to (C8);
				\draw[] (B9) to (C9);
				\node[right=of C1](C11){};
				\node[right=of C11](C111){};
				\node[right=of C111](C1111){};
				\node[right=of C1111](C11111){};
				\node[right=of C11111](C111111){};
				\node[right=of C111111, circle,draw, inner sep=1pt, label=left:{\Large$1_1$}](1C1){};
				\node[below=of 1C1](1C11){};
				\node[below=of 1C11, circle,draw, inner sep=1pt, label=left:{\Large$(h_1^{-1},1)$}](1C2){};
				\node[below=of 1C2](C21){};
				\node[below=of C21, circle,draw, inner sep=1pt, label=left:{\Large$(h_2^{-1},1)$}](1C3){};
				\node[below=of 1C3](C31){};
				\node[below=of C31, circle,draw, inner sep=1pt, label=left:{\Large$(h_3^{-1},1)$}](1C4){};
				\node[below=of 1C4, circle,draw, inner sep=1pt, label=left:{}](1C5){};
				\node[below=of 1C5, circle,draw, inner sep=1pt, label=left:{}](1C6){};
				\node[below=of 1C6, circle,draw, inner sep=1pt, label=left:{}](1C7){};
				\node[below=of 1C7, circle,draw, inner sep=1pt, label=left:{\Large$(h_{t-1}^{-1},1)$}](1C8){};
				\node[below=of 1C8](C81){};
				\node[below=of C81, circle,draw, inner sep=1pt, label=left:{\Large$(h_{t}^{-1},1)$}](1C9){};
				\node[right=of 1C1](C00){};
				\node[right=of C00](C000){};
				\node[right=of C000](C0000){};
				\node[right=of C0000, circle,draw, inner sep=1pt, label=right:{\Large$1_3$}](D1){};
				\node[below=of D1](D11){};
				\node[below=of D11, circle,draw, inner sep=1pt, label=right:{\Large$(h_1,3)$}](D2){};
				\node[below=of D2](D21){};
				\node[below=of D21, circle,draw, inner sep=1pt, label=right:{\Large$(h_2h_1^{-1},3)$}](D3){};
				\node[below=of D3](D31){};
				\node[below=of D31, circle,draw, inner sep=1pt, label=above:{\Large$(h_3h_2^{-1},3)$}](D4){};
				\node[below=of D4, circle,draw, inner sep=1pt, label=above:{}](D5){};
				\node[below=of D5, circle,draw, inner sep=1pt, label=above:{}](D6){};
				\node[below=of D6, circle,draw, inner sep=1pt, label=above:{}](D7){};
				\node[below=of D7, circle,draw, inner sep=1pt, label=right:{\Large$(h_{t-1}h_{t-2}^{-1},3)$}](D8){};
				\node[below=of D8](D81){};
				\node[below=of D81, circle,draw, inner sep=1pt, label=right:{\Large$(h_{t}h_{t-1}^{-1},3)$}](D9){};
				\node[below=of D9](D91){};
				\node[below=of D91,label=below:{\LARGE$[\Gamma_4(1_2)]$}]{};
				\draw[] (1C1) to (D1);
				\draw[] (1C1) to (D2);
				\draw[] (1C1) to (D3);
				\draw[] (1C2) to (D1);
				\draw[dashed](1C2) to (D2);
				\draw[dashed](1C3) to (D2);
				\draw[dashed] (1C3) to (D3);
				\draw[] (1C3) to (D4);
				\draw[] (1C4) to (D1);
				\draw (1C4) to (D5);
				\draw (1C5) to (D1);
				\draw (1C7) to (D1);
				\draw (1C7) to (D8);
				\draw[] (1C8) to (D1);
				\draw[dashed] (1C8) to (D8);
				\draw[] (1C8) to (D9);
				\draw[] (1C9) to (D1);
				\draw[dashed] (1C9) to (D9);
				\node[right=of D1](D11){};
				\node[right=of D11](D111){};
				\node[right=of D111](D1111){};
				\node[right=of D111](D1111){};
				\node[right=of D1111](D11111){};
				\node[right=of D11111](D111111){};
				\node[right=of D111111](D1111111){};
				\node[right=of D11111, circle,draw, inner sep=1pt, label=right:{\Large$(h_1h_2h_3h_4,4)$}](E1){};
				\node[below=of E1](E11){};
				\node[below=of E11, circle,draw, inner sep=1pt, label=right:{\Large$(h_1h_3h_4,4)$}](E2){};
				\node[below=of E2](E21){};
				\node[below=of E21, circle,draw, inner sep=1pt, label=right:{\Large$(h_2h_3h_4,4)$}](E3){};
				\node[below=of E3](E31){};
				\node[below=of E31, circle,draw, inner sep=1pt, label=right:{\Large$(h_3h_2h_1^{-1},4)$}](E4){};
				\node[below=of E4, circle,draw, inner sep=1pt, label=right:{}](E5){};
				\node[below=of E5, circle,draw, inner sep=1pt, label=right:{}](E6){};
				\node[below=of E6, circle,draw, inner sep=1pt, label=right:{}](E7){};
				\node[below=of E7, circle,draw, inner sep=1pt, label=right:{\Large$(h_{t-1}h_2h_1^{-1},4)$}](E8){};
				\node[below=of E8](E81){};
				\node[below=of E81, circle,draw, inner sep=1pt, label=right:{\Large$(h_{t}h_2h_1^{-1},4)$}](E9){};
				\draw[] (D3) to (E4);
				\draw[] (D3) to (E5);
				\draw[] (D3) to (E6);
				\draw[] (D3) to (E8);
				\draw[] (D3) to (E9);
				\draw[dashed](D4) to (E3);
				\draw[dashed](D4) to (E4);
				\draw[dashed] (D5) to (E2);
				\draw[dashed] (D5) to (E3);
			\end{tikzpicture}
			\vspace{3mm}
		\end{figure}
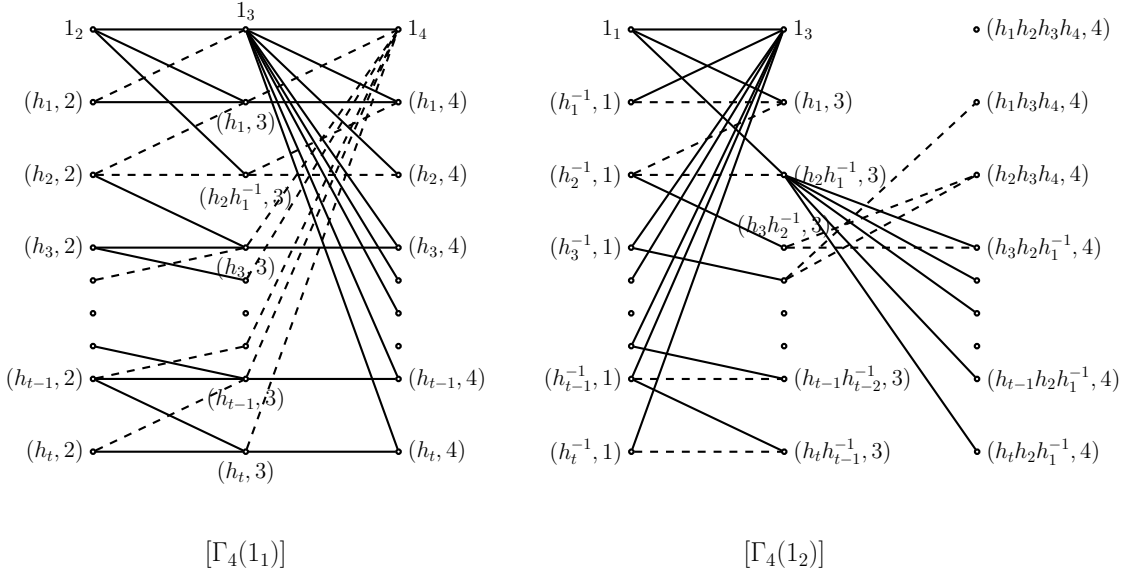
		
		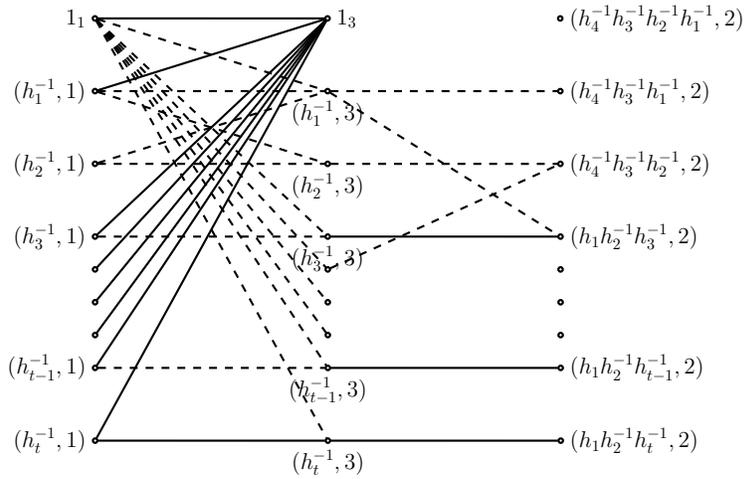
\begin{figure}[!ht]
			\vspace{3mm}
			\caption{The induced subgraphs $[\Gamma_4(1_4)]$}\label{Fig62}
			\begin{tikzpicture}[node distance=0.6cm,thick,scale=0.6,every node/.style={transform shape},scale=1]
				\node[circle](A0){};
				\node[right=of A0](AA){};
				\node[right=of AA](AAA){};
				\node[right=of AAA](AAAA){};
				\node[right=of AAAA](AAAAA){};
				\node[right=of AAAAA](AAAAAA){};
				\node[right=of AAAAAA](AAAAAAA){};
				\node[right=of AAAAAAA](AAAAAAA1){};
				\node[right=of AAAAAAA1](AAAAAAA2){};
				\node[right=of AAAAAAA2](AAAAAAA3){};
				\node[below=of AAAAAAA1, circle,draw, inner sep=1pt, label=left:{\Large$1_1$}](A1){};
				\node[below=of A1](A11){};
				\node[below=of A11, circle,draw, inner sep=1pt, label=left:{\Large$(h_1^{-1},1)$}](A2){};
				\node[below=of A2](A21){};
				\node[below=of A21, circle,draw, inner sep=1pt, label=left:{\Large$(h_2^{-1},1)$}](A3){};
				\node[below=of A3](A31){};
				\node[below=of A31, circle,draw, inner sep=1pt, label=left:{\Large$(h_3^{-1},1)$}](A4){};
				\node[below=of A4, circle,draw, inner sep=1pt, label=left:{}](A5){};
				\node[below=of A5, circle,draw, inner sep=1pt, label=left:{}](A6){};
				\node[below=of A6, circle,draw, inner sep=1pt, label=left:{}](A7){};
				\node[below=of A7, circle,draw, inner sep=1pt, label=left:{\Large$(h_{t-1}^{-1},1)$}](A8){};
				\node[below=of A8](A81){};
				\node[below=of A81, circle,draw, inner sep=1pt, label=left:{\Large$(h_{t}^{-1},1)$}](A9){};
				\node[right=of A1](A00){};
				\node[right=of A00](A000){};
				\node[right=of A000](A0000){};
				\node[right=of A0000](A00000){};
				\node[right=of A00000](A000000){};
				\node[right=of A000000, circle,draw, inner sep=1pt, label=right:{\Large$1_3$}](B1){};
				\node[below=of B1](B11){};
				\node[below=of B11, circle,draw, inner sep=1pt, label=below:{\Large$(h_1^{-1},3)$}](B2){};
				\node[below=of B2](B21){};
				\node[below=of B21, circle,draw, inner sep=1pt, label=below:{\Large$(h_2^{-1},3)$}](B3){};
				\node[below=of B3](B31){};
				\node[below=of B31, circle,draw, inner sep=1pt, label=below:{\Large$(h_3^{-1},3)$}](B4){};
				\node[below=of B4, circle,draw, inner sep=1pt, label=below:{}](B5){};
				\node[below=of B5, circle,draw, inner sep=1pt, label=below:{}](B6){};
				\node[below=of B6, circle,draw, inner sep=1pt, label=below:{}](B7){};
				\node[below=of B7, circle,draw, inner sep=1pt, label=below:{\Large$(h_{t-1}^{-1},3)$}](B8){};
				\node[below=of B8](B81){};
				\node[below=of B81, circle,draw, inner sep=1pt, label=below:{\Large$(h_{t}^{-1},3)$}](B9){};
				\draw[] (A1) to (B1);
				\draw[dashed] (A1) to (B2);
				\draw[dashed] (A1) to (B4);
				\draw[dashed] (A1) to (B5);
				\draw[dashed] (A1) to (B6);
				\draw[dashed] (A1) to (B7);
				\draw[dashed] (A1) to (B8);
				\draw[dashed] (A1) to (B9);
				\draw[] (A2) to (B1);
				\draw[dashed] (A2) to (B2);
				\draw[dashed] (A2) to (B3);
				\draw[dashed] (A3) to (B2);
				\draw[dashed] (A3) to (B3);
				\draw[] (A4) to (B1);
				\draw[] (A5) to (B1);
				\draw[] (A6) to (B1);
				\draw[] (A7) to (B1);
				\draw[dashed] (A4) to (B4);
				\draw[] (A8) to (B1);
				\draw[dashed] (A8) to (B8);
				\draw[] (A9) to (B9);
				\draw[] (A9) to (B1);
				\node[right=of B1](B11){};
				\node[right=of B11](B111){};
				\node[right=of B111](B1111){};
				\node[right=of B1111](B11111){};
				\node[right=of B11111](B111111){};
				\node[right=of B111111, circle,draw, inner sep=1pt, label=right:{\Large$(h_4^{-1}h_3^{-1}h_2^{-1}h_1^{-1},2)$}](C1){};
				\node[below=of C1](C11){};
				\node[below=of C11, circle,draw, inner sep=1pt, label=right:{\Large$(h_4^{-1}h_3^{-1}h_1^{-1},2)$}](C2){};
				\node[below=of C2](C21){};
				\node[below=of C21, circle,draw, inner sep=1pt, label=right:{\Large$(h_4^{-1}h_3^{-1}h_2^{-1},2)$}](C3){};
				\node[below=of C3](C31){};
				\node[below=of C31, circle,draw, inner sep=1pt, label=right:{\Large$(h_1h_2^{-1}h_3^{-1},2)$}](C4){};
				\node[below=of C4, circle,draw, inner sep=1pt, label=right:{}](C5){};
				\node[below=of C5, circle,draw, inner sep=1pt, label=right:{}](C6){};
				\node[below=of C6, circle,draw, inner sep=1pt, label=right:{}](C7){};
				\node[below=of C7, circle,draw, inner sep=1pt, label=right:{\Large$(h_1h_2^{-1}h_{t-1}^{-1},2)$}](C8){};
				\node[below=of C8](C81){};
				\node[below=of C81, circle,draw, inner sep=1pt, label=right:{\Large$(h_1h_2^{-1}h_t^{-1},2)$}](C9){};
				\draw[dashed] (B2) to (C2);
				\draw[dashed] (B2) to (C4);
				\draw[] (B4) to (C4);
				\draw[dashed] (B3) to (C3);
				\draw[dashed] (B5) to (C3);
				\draw[] (B8) to (C8);
				\draw[] (B9) to (C9);
				\node[below=of B9](B90){};
			\end{tikzpicture}
			\vspace{3mm}
		\end{figure}
		
		\smallskip
		\textsf{Step 1}: prove that $A$ stabilizes $G_i$ for each $i\in\{1,2,3,4\}$.
		\smallskip
		
		Consider $[\Delta(1_3)]$.
		Since $T_{1,2}=T_{1,4}=S$, it follows that $\{1_1,1_2\}$, $\{1_1,1_4\}$, $\{(h_i^{-1},1),1_2\}$ with $i\in \{1,3,\ldots,t\}$ and $\{(h_i^{-1},1),1_4\}$ with $i\in \{1,3,\ldots,t\}$ are edges in $[\Delta(1_3)]$. In particular, $1_2$ and $1_4$ have at least $t$ common neighbors $1_1$ and $(h_i^{-1},1)$ with $i\in\{1,3\ldots,t\}$ in $[\Delta(1_3)]$. However, as $G\neq C_2^4$, we observe for any $i\in \{1,2,4\}$ that the graph $[\Delta(1_i)]$ has no two vertices with at least $t$ common neighbors in $[\Delta(1_i)]$. Thus $A$ stabilizes $G_3$.
		Now the vertex $1_3$ has valency at least $t+1$ in $[\Delta(1_1)]$, while for $i\in\{2,4\}$, no vertex from $G_3$ has valency at least $t+1$ in $[\Delta(1_i)]$ (again as $G\neq C_2^4$). Hence $A$ stabilizes $G_1$.
		
		To finish Step~1, it remains to prove that $A$ stabilizes each of $G_2$ and $G_4$. Suppose for a contradiction that $(1_2)^{\a}=1_4$ for some $\a\in A$. Since $1_2$ and $1_4$ have valencies $3t+3$ and $3t+2$, respectively, in $\Sigma_4$, it follows that $\Delta=\Gamma_4$.
		In the graph $[\Delta(1_2)]$, the vertex $(h_2h_1^{-1},3)$ has $t-2$ neighbors in $G_4$. Then since $[\Delta(1_2)]^{\a}=[\Delta(1_4)]$ and since $\a$ stabilizes $G_1$ and $G_3$ respectively, the graph $[\Delta(1_4)]$ must have some vertex in $G_3$ with at least $t-2$ neighbors in $G_2$. This forces $t=4$ and that either $(h_1^{-1},3)$ or $(h_4^{-1},3)$ is a vertex of $[\Delta(1_4)]$ with two neighbors in $G_2$.
		
		First assume that $(h_1^{-1},3)$ is a vertex of $[\Delta(1_4)]$ with two neighbors in $G_2$. Since
		\[
		\Delta((h_1^{-1},3))\cap\Delta(1_4)\cap G_2
		=\Delta((h_1^{-1},3))\cap(N^{-1}\times{2})\subseteq\{h_4^{-1}h_3^{-1}h_1^{-1},h_1h_2^{-1}h_3^{-1}\}\times\{2\},
		\]
		it follows that $(h_4^{-1}h_3^{-1}h_1^{-1},2)\in\Delta((h_1^{-1},3))$, which is equivalent to $h_4h_3^{-1}=h_3h_4$.
		In this case, $\{(h_2^{-1},3),((h_4^{-1}h_3^{-1}h_2^{-1},2)\}$ is also an edge in $[\Delta(1_4)]$, and so $1_3$ is the only vertex of $[\Delta(1_4)]$ in $G_3$ with no neighbor in $G_2$.
		However, there are two vertices $1_3$ and $(h_1,3)$ of $[\Delta(1_2)]$ in $G_3$ with no neighbor in $G_4$, a contradiction.
		
		Now we conclude that $(h_4^{-1},3)$ is the only vertex of $[\Delta(1_4)]$ with two neighbors in $G_2$. As a consequence, $(h_2h_1^{-1},3)^{\a}=(h_4^{-1},3)$. Notice that, for $i\in\{2,4\}$, the only vertex in $G_3$ with valency $t$ in $[\Delta(1_i)]$ is $1_3$. We deduce that $(1_3)^{\a}=1_3$ and hence
		\[
		\big(\Delta(1_3)\cap\Delta((h_2h_1^{-1},3))\cap\Delta(1_2)\big)^\a=\Delta(1_3)\cap\Delta(h_4^{-1},3)\cap\Delta(1_4).
		\]
		Since $\Delta(1_3)\cap\Delta((h_2h_1^{-1},3))\cap\Delta(1_2)=\{1_1\}$, it follows that $|\Delta(1_3)\cap\Delta(h_4^{-1},3)\cap\Delta(1_4)|=1$. This together with the observation $(h_4^{-1},1)\in\Delta(1_3)\cap\Delta(h_4^{-1},3)\cap\Delta(1_4)$ implies that $1_1\notin\Delta((h_4^{-1},3))$, and so in $[\Delta(1_4)]$ the vertex $(h_4^{-1},3)$ has no neighbor of valency $3$. However, in $[\Delta(1_2)]$, the vertex $1_1$ is a neighbor of $(h_2h_1^{-1},3)$ with valency $3$. This contradiction completes Step~1.
		
		\smallskip
		\textsf{Step 2}: prove $A_{1_1}=1$.
		\smallskip
		
		Since $[G_1\cup G_2\cup G_3]\cong\Gamma_3$ and since Lemma~\ref{lem=4.1} asserts that $\Gamma_3$ is a $3$-HGR of $G$, we conclude that $A_{1_1}$ fixes $G_1\cup G_2\cup G_3$ pointwise.
		Suppose for a contradiction that $A_{1_1}\neq1$. Then there exists $\a\in A_{1_1}$ such that $(g,4)^{\a}=(h,4)$ for some distinct $g$ and $h$ in $G$. It follows that $\a$ maps the neighbors of $(g,4)$ in $G_1$ to the neighbors of $(h,4)$ in $G_1$, that is, $((S^{-1}g)\times\{1\})^{\a}=(S^{-1}h)\times\{1\}$.
		Since $\a$ fixes $G_1$ pointwise, this yields $(S^{-1}g)\times\{1\}=(S^{-1}h)\times\{1\}$ and hence $S^{-1}g=S^{-1}h$. In particular, $g\in S^{-1}h$, which means that $g=h_i^{-1}h$ for some $i\in\{1,\ldots,t\}$ (as $g\neq h$). However, this leads to $S^{-1}h_i^{-1}h=S^{-1}h$, or equivalently, $h_iS=S$, which is impossible. Thus $A_{1_1}=1$, as required.
	\end{proof}
	
	We now prove Theorem~\ref{theo=main} in the case when $m$ is even, which together with the end of Section~\ref{Sec5} completes the proof of Theorem~\ref{theo=main}.
	
	\begin{proof}[Proof of Theorem~$\ref{theo=main}$ for even $m$]
		Propositions~\ref{prop=smallgroups} and~\ref{prop=smallranks} establish that Theorem~\ref{theo=main} holds for $d(G)\leq3$. Assume that $d(G)\geq4$. Then as Lemma~\ref{Lem6.3} asserts, $G$ has a $4$-HGR $\Gamma_4$ and a $4$-PGSR $\Sigma_4$ satisfying the conditions of Lemma~\ref{Lem4PGSR}. The latter, in turn, implies that $G$ has an $m$-HGR for each even $m\geq6$. Hence Theorem~\ref{theo=main} holds for every even $m\geq4$.
	\end{proof}

	\noindent{\bf Acknowledgements:} The first author was supported by the National Natural Science Foundation of China (12101601, 12271024, 12371025) and the National Key R\&D Program of China (211070B62501). The second author was partially supported by the National Natural Science Foundation of China (12331013, 12271024). The fourth author was supported by the National Natural Science Foundation of China (12101070, 12161141005).

\end{document}